\documentclass[a4paper,reqno]{amsart}

\usepackage[utf8]{inputenc}
\usepackage{amssymb}
\usepackage{enumitem}
\usepackage{mathrsfs}
\usepackage{mathtools}
\usepackage{amsmath}
\usepackage{xcolor}
\usepackage[colorlinks=true]{hyperref}
\hypersetup{linkcolor=violet,urlcolor=cyan,citecolor=cyan}

\setlist[enumerate]{label=\emph{(\roman*)}}

\usepackage{environ}

\newtheorem{theorem}{Theorem}[section]

\newtheorem{lemma}[theorem]{Lemma}
\newtheorem{proposition}[theorem]{Proposition}

\theoremstyle{definition}

\newtheorem{remark}[theorem]{Remark}
\numberwithin{equation}{section}
\newcommand{\R}{\mathbb{R}}

\def \d {{\rm{d}}}

\begin{document}
	
	\parindent=0pt
	
	\title[Uniform boundedness of conformal energy]
	{Uniform boundedness of conformal energy for the 3D nonlinear wave equation}
	
	\author[J. Zhao]{Jingya Zhao}
	\address{Southern University of Science and Technology, Department of Mathematics, 518055 Shenzhen, China.}
	\email{12231282@mail.sustech.edu.cn}
	\begin{abstract}
		In this paper, we study three-dimensional nonlinear wave equations under the null condition, a fundamental model in the theory of nonlinear wave-type equations, initially investigated by Christodoulou \cite{Christodoulou86} and Klainerman \cite{Klainerman86}. For a class of large initial data, we establish global existence and linear scattering of solutions by combining refined energy estimates with a bootstrap argument. Moreover, we prove that the lower-order conformal energy remains uniformly bounded for all time.
	\end{abstract}

	\maketitle
	
	\section{Introduction}
	
	\subsection{Model problem}
	We consider the following nonlinear wave equation in $\R^{1+3}$, which reads
	\begin{equation}\label{equ:Wave}
		-\Box w = P^{\gamma\alpha \beta} \partial_\gamma w\partial_\alpha  \partial_\beta w+\partial^\alpha w\partial_\alpha w,
	\end{equation}
	with the initial data prescribed at $t=0$
	\begin{equation}
		(w,\partial_t w)|_{t=0}=(w_0,w_1).\label{est:initial}
	\end{equation}
	
	In the above, $\Box = m^{\alpha\beta}\partial_\alpha \partial_\beta = -\partial_t^2 + \partial_1^2 + \partial_2^2 + \partial_3^2$ denotes the wave operator associated with the Minkowski metric $m=\mathrm{diag}(-1,1,1,1)$. Indices are raised and lowered with respect to $m$. The coefficients $P^{\gamma\alpha\beta}=P^{\gamma\beta\alpha}$ are constants satisfying the null condition
	\[
	P^{\gamma\alpha\beta}\xi_\gamma \xi_\alpha \xi_\beta = 0, \qquad \forall\, \xi \in \mathbb{R}^4 \text{ with } \xi_0^2=\sum_{a=1}^3 \xi_a^2.
	\]
	The null condition, first introduced by Klainerman~\cite{Klainerman83,Klainerman86}, arises naturally in various nonlinear wave equations. Throughout this paper, we adopt the Einstein summation convention for repeated upper and lower indices. Greek indices $\alpha,\beta,\ldots$ range over $\{0,1,2,3\}$, while Roman indices $a,b,\ldots$ range over $\{1,2,3\}$.
	
	\subsection{Main results}
	In this work, we investigate the nonlinear wave equation~\eqref{equ:Wave} with a class of large initial data. We establish the global existence of solutions together with linear scattering, prove the uniform boundedness of both the natural (ghost weight) energy and the lower-order conformal energy, and derive pointwise decay estimates. The main results are summarized below.
	\begin{theorem}\label{thm:main1}
		Let $N\in\mathbb{N}$ with $N\geq5$. For any $K>1$, there exists an $\varepsilon_0>0$, depending polynomially on $K$, such that for all initial data $(w_0,w_1)$ satisfying
		\begin{align}
			&\sum_{0\leq |I|\leq N}\|\langle x\rangle^{|I|}\nabla\nabla^Iw_0\|_{L^2}+\sum_{0\leq |I|\leq N}\|\langle x\rangle^{|I|}\nabla^Iw_1\|_{L^2}\leq K,\label{est:initial1}\\
			&\sum_{0\leq |I|\leq N-1}\|\langle x\rangle^{|I|}\nabla\nabla\nabla^Iw_0\|_{L^2}+\sum_{0\leq |I|\leq N-1}\|\langle x\rangle^{|I|}\nabla\nabla^Iw_1\|_{L^2}\leq\varepsilon<\varepsilon_0,\label{est:initial2}
		\end{align}
	the Cauchy problem~\eqref{equ:Wave}--\eqref{est:initial} admits a global-in-time solution $w$ satisfying
	\begin{equation}
		|\partial w(t,x)|\leq C\langle t+|x|\rangle^{-1}\langle t-|x|\rangle^{-\frac{1}{2}},\quad
		|\partial\partial w(t,x)|\leq C\varepsilon\langle t+|x|\rangle^{-1}\langle t-|x|\rangle^{-\frac{1}{2}},
	\end{equation}
    and the ghost weight energy is uniformly bounded, i.e.,
    \begin{equation}\label{est:ghostenergy}
    	\sum_{|I|\leq N}\mathcal{E}_{gst}^{\frac{1}{2}}(t,\Gamma^Iw)\leq C,
    \end{equation}
    for some constant $C$ independent of $t$. In addition, the solution $w$ scatters linearly.
	\end{theorem}
	
	\begin{remark}
		Nontrivial initial data $(w_0,w_1)$ satisfying the conditions~\eqref{est:initial1} and~\eqref{est:initial2} do exist. 
		For instance, one may take the initial data of the form~\cite{DoLiZhao23}
		\begin{equation*}
			\nabla w_0(x)=\frac{f_{\varepsilon}(x)+\varepsilon f(x)}{\|f_{\varepsilon}(x)+\varepsilon f(x)\|_{L^2}},\qquad 
			w_1(x)=\frac{g_{\varepsilon}(x)+\varepsilon g(x)}{\|g_{\varepsilon}(x)+\varepsilon g(x)\|_{L^2}},
		\end{equation*}
		where $f,g$ are arbitrary smooth functions with sufficiently fast decay at spatial infinity, and
		\[
		f_{\varepsilon}(x)=\varepsilon^{\frac{3}{2}}f(\varepsilon x), \qquad 
		g_{\varepsilon}(x)=\varepsilon^{\frac{3}{2}}g(\varepsilon x).
		\]
		Such a choice ensures that the weighted norms in~\eqref{est:initial1} remain bounded by $K$, while those in~\eqref{est:initial2} can be made arbitrarily small by choosing $\varepsilon>0$ sufficiently small.
	\end{remark}

    \begin{theorem}\label{thm:main2}
    	In addition to the assumptions of Theorem~\ref{thm:main1}, suppose further that the initial data 
    	$(w_0,w_1)$ satisfy
    	\begin{equation}\label{est:initialconf}
    		\begin{aligned}
    			\sum_{0\leq |I|\leq N}&\|\langle x\rangle^{|I|}\nabla^{I}w_0\|_{L^2}+\sum_{0\leq |I|\leq N-1}\|\langle x\rangle^{|I|+1}\nabla^{I}w_1\|_{L^2}\\
    			&+\sum_{1\leq |I|\leq N-1}\|\langle x\rangle^{|I|}\nabla\nabla^{I}w_0\|_{L^1}+\sum_{0\leq |I|\leq N-1}\|\langle x\rangle^{|I|}\nabla^{I}w_1\|_{L^1}<+\infty.
    		\end{aligned}
    	\end{equation}
    	Then the corresponding global solution $w$ given by Theorem~\ref{thm:main1}$\,$ enjoys the improved pointwise decay estimates
    	\begin{equation}
    		\left|w(t,x)\right|\leq C\langle t+|x|\rangle^{-1}\langle t-|x|\rangle^{-\frac{1}{2}},\quad 
    		\left|\partial w(t,x)\right|\leq C\langle t+|x|\rangle^{-1}\langle t-|x|\rangle^{-\frac{3}{2}},
    	\end{equation}
    	and the conformal energy remains uniformly bounded for all $t\geq 0$, i.e.,
    	\begin{equation}\label{est:conenerg}
    		\sum_{|I|\leq N-3}\mathcal{E}_{con}^{\frac{1}{2}}(t,\Gamma^Iw)\leq C,
    	\end{equation}
    	for some constant $C$ independent of $t$.
    \end{theorem}

	\begin{remark}
		
		The main contribution of this paper is establishing the uniform boundedness of the lower-order conformal energy in~\eqref{est:conenerg}. We recall that the natural (ghost weight) energy is defined as
		$$ \mathcal{E}_{gst}(t,w)=\int_{\R^{3}}((\partial_t w)^2+\sum_{a=1}^{3}(\partial_a w)^2)(t,x)\d x+\sum_{a=1}^{3}\int_{0}^{t}\int_{\R^{3}}\frac{|G_aw|^2}{\langle r-s\rangle^{1+2\delta}}\d x\d s,$$
		and the conformal energy is defined as
		$$ \mathcal{E}_{con}(t,w)=\int_{\R^{3}}((L_0 w)^2+w^2+\sum_{1\leq a<b\leq3}|\Omega_{ab} w|^2+\sum_{a=1}^{3}|L_a w|^2)(t,x)\d x,$$
		and one finds more details about the vector fields $L_0, \Omega_{ab}, L_a\ \mbox{and} \ G_a$ in Section \ref{Se:Pre}.
		We note, in general, it is more challenging to bound the conformal energy than the natural energy for wave equations (see~\eqref{est:ghost} and~\eqref{est:Conformal}), and thus it is a nontrivial task to establish the uniform boundedness of the conformal energy in~\eqref{est:conenerg}.
	\end{remark}
	
	\subsection{Relevant results}
	Many physical models can be described by nonlinear wave equations, such as Dirac equations, Einstein equations, and Maxwell equations. Whether a wave equation admits a global solution is a central question in this field, and it is therefore of vital importance to explore the long-time behavior of such solutions.
	In 1986, the global existence of solutions to 3D quasilinear wave equations with the null condition was first proved independently by Christodoulou \cite{Christodoulou86} and Klainerman \cite{Klainerman86}. Their pioneering results were obtained through different approaches: Christodoulou used the conformal mapping method, while Klainerman employed the vector field method \cite{Klainerman85}. In 1993, Christodoulou and Klainerman \cite{Christo-Klainer1993} first established a landmark theory on the global nonlinear stability of Minkowski space. Later, Lindblad and Rodnianski \cite{LindbladRodnianski2010} provided a new proof of the global stability of Minkowski spacetime in harmonic gauge.
	
	Due to the slow decay of the waves, the two-dimensional case is substantially more delicate. The problem of global existence for quasilinear wave equations with the null condition in two space dimensions remained open until 2001, except for the cases of cubic terms or rotational invariance. Alinhac \cite{Alinhac2001} obtained the global existence for 2D quasilinear wave equations with the null condition by constructing an approximate solution and the ghost weight method. However, the result in \cite{Alinhac2001} only provided an upper bound with polynomial growth in time for the top-order energy of solutions. Motivated by this, Cai, Lei, and Masmoudi \cite{Cai} in 2018 proved the uniform boundedness result for a class of quasilinear nonlinearities satisfying the strong null condition in the sense of Lei \cite{Lei}. Subsequently, Dong, LeFloch, and Lei~\cite{DoLeLe21} extended this result and established the uniform boundedness of the top-order energy in \cite{Alinhac2001} for general null nonlinearities by employing the hyperboloidal vector field method. More recently, Li \cite{Li21} developed a new strategy that avoids the use of Lorentz boost vector fields and proved the uniform boundedness of the highest order norm of the solution to the two-dimensional quasilinear wave equation with standard quadratic null-type nonlinearities. 
	
	On the other hand, the proof in \cite{Alinhac2001} required that the initial data have compact support. This restriction was later removed: in 2018, Cai, Lei, and Masmoudi \cite{Cai} proved global well-posedness for a 2D fully nonlinear wave equation with small but non-compactly supported initial data, and extended their result to a class of 2D quasilinear equations. In 2020, Hou and Yin \cite{Yin} first obtained global existence for general null quasilinear wave equations with non-compactly supported initial data in two dimensions by establishing a new class of weighted $L^\infty$-$L^\infty$ estimates. 
	
	Relevant results for the 1-dimensional case can be referred to \cite{LiMengni,Luli}. When the initial data are large, one can refer to \cite{Ding,Yu}.
	
	There are also various results concerning generalizations of the classical null condition. In 2003, Alinhac \cite{Alinhac2003} studied global existence for an example of 3D quasilinear wave equations violating the null condition (see also \cite{Lindblad2008}). Later, Lindblad and Rodnianski \cite{LindbladRodnianski2010} came up with the weak null condition, which generalizes the null condition; see also the KMS condition in \cite{Katayama2015} by Katayama, Matoba and Sunagawa. Another related generalization is the non-resonance condition introduced by Pusateri and Shatah \cite{PusateriShatah2013}.
	Keir \cite{Keir18} established global existence for a large class of quasilinear wave equations under the weak null condition, by applying the $p$-weighted energy method developed in \cite{DafermosRodnianski:rp}. His approach also applies to the coupled Einstein-Maxwell system in harmonic coordinates and Lorenz gauge, as well as to various other scalar wave equations. Very recently,  Dong, Li, and Yuan \cite{DoLiYu22} established the uniform boundedness of the total energy for a massless Dirac-Klein-Gordon system in $(1+3)$-dimensional Minkowski spacetime, which can be regarded as a weak null wave system. Their analysis relies on a new weighted conformal energy estimate that compensates for the slow decay of solutions. Some further results about the weak null condition can be referred to \cite{Katayama2015,Yoko,DengPusateri2020}.
	
	The wave system with multispeeds is also a topic of concern to many researchers. The global existence of solutions in nonlinear elastodynamics was established by Agemi \cite{Agemi2000} and Sideris \cite{Sideris2000}. Yokoyama \cite{Yokoyama2000} studied the global existence for three-dimensional systems of quasilinear wave equations with multispeeds. The global existence for 3D nonlinear wave equations with multispeeds under the null condition can be referred to \cite{Sideris2001}. 
	
	Last but not least, the precise pointwise behaviors for the wave equations have also drawn much attention. In 2020, Deng and Pusateri \cite{DengPusateri2020} considered a wave equation satisfying the weak null condition and showed the precise pointwise behavior of the differentiated solution near the light cone. Recently, Dong-Ma-Ma-Yuan \cite{DMMY22} demonstrated the precise behavior of the solution to \eqref{equ:Wave} with initial data prescribed on a hyperboloid.

	\subsection{Outline of the proof}
	In this subsection, we present the key points of the proof in the paper. It is well known that the global existence and uniform boundedness of the natural energy for small-data solutions to the Cauchy problem ~\eqref{equ:Wave}--\eqref{est:initial} have already been established in the literature~\cite{Wang,MaHu17}. In this paper, we obtain the global existence and the uniform boundedness of natural energy by using the vector field theory and bootstrap argument. The ghost weight method of Alinhac plays an important role in the proof of the uniform boundedness estimate on the higher-order natural energy (see also \cite{Wang}). A key challenge in establishing Theorem~\ref{thm:main1} is handling the nonlinearities induced by the large initial data.
	
	The most difficult part in Theorem \ref{thm:main2} lies in showing the uniform boundedness of the lower-order conformal energy. More specifically, we know that there is a $\langle t\rangle^{-1}$ decay which is critical in the estimate on the conformal energy (see (\ref{est:Conformal})). Therefore, we need to find a new way to prove the lower-order conformal energy is uniformly bounded. To achieve this, we apply the Fourier method to obtain a new $L^2$ estimate of the undifferentiated solutions to the wave equations (see (\ref{est:uL2})). Then we can obtain the higher-order weighted estimates on the solution by commuting the wave equation with the vector fields. Based on these estimates, we in turn complete the proof of uniform boundedness of  conformal energy up to order $N-3$. See Section \ref{Se:Con-energy} for more details.
	
	\subsection{Organization of the paper}
	The paper is organized as follows. In Section~\ref{Se:Pre}, we introduce the notation and basic estimates that will be used in Sections~\ref{Se:exist} and~\ref{Se:Con-energy}. In Section~\ref{Se:exist}, we establish the global existence and linear scattering of the solution to the Cauchy problem~\eqref{equ:Wave}--\eqref{est:initial}, and obtain the uniform boundedness of the ghost weight energy. Section~\ref{Se:Con-energy} is devoted to proving that the lower-order conformal energy is also uniformly bounded.
	
	\section{Notation and basic estimates}\label{Se:Pre}
	\subsection{Notation}
	Our problem is set in the $(1+3)$-dimensional spacetime $\R^{1+3}$. A spacetime point is denoted by $(t,x)=(x_{0},x_{1},x_{2},x_{3})$, with spatial radius $r=\sqrt{x_{1}^{2}+x_{2}^{2}+x_{3}^{2}}$. We set $\omega_{a}=\frac{x_{a}}{r}$ for $a=1,2,3$, and write $x=(x_{1},x_{2},x_{3})\in \R^{3}$. Spacetime indices are represented by Greek letters $\{\alpha,\beta,\dots\}$, while spatial indices are denoted by Roman letters $\left\{a,b,\dots\right\}$. 
	
	The Japanese bracket is defined by $\langle \rho\rangle=(1+|\rho|^{2})^{\frac{1}{2}}$ for $\rho\in \R$ (and more generally $\rho \in \R^3$). We write $A \lesssim B$ to mean that $A \leq C_0 B$ for some universal constant $C_0$. Similarly, $A \sim B$ denotes $A \lesssim B$ and $B \lesssim A$. Unless otherwise specified, we adopt the Einstein summation convention for repeated lower and upper indices. In what follows, we omit the subscript in $\|f\|_{L^2_x}$ whenever no confusion arises, for $f\in L^2(\mathbb{R}^3)$.
	
	Following Klainerman~\cite{Klainerman86}, we introduce the standard set of vector fields in $\R^{1+3}$
	\begin{enumerate}
		\item [(1)] Translations: $\partial_{\alpha}:=\partial_{x_{\alpha}}$, for $\alpha=0,1,2,3$.
		\item [(2)] Rotations: $\Omega_{ab}:=x_{a}\partial_{b}-x_{b}\partial_{a}$, for $1\leq a<b\leq3$.
		\item [(3)] Scaling vector field: $L_{0}:=t\partial_t +x^{a}\partial_{a}$.
		\item [(4)] Lorentz boosts: $L_{a}:=t\partial_{a}+x_{a}\partial_{t}$, for $a=1,2,3$.
	\end{enumerate}
	
	We define the ordered set of vector fields
	\begin{equation*}
			\Gamma=\left(\Gamma_{1},\dots,\Gamma_{11}\right)=\left(\partial,\Omega,L,L_0\right),
	\end{equation*}
	where 
	\begin{equation*}
		\begin{aligned}
			\partial=(\partial_0,\partial_1,\partial_2,\partial_3)=(\partial_t,\nabla),\ \ \Omega=(\Omega_{12},\Omega_{13},\Omega_{23}), \ \  L=(L_1,L_2,L_3).
		\end{aligned}
	\end{equation*}
Moreover, for any multi-index $I=(I_{1},\dots,I_{11})\in \mathbb{N}^{11}$, we denote
	$$ \Gamma^{I}=\prod_{k=1}^{11}\Gamma_{k}^{I_{k}}, \quad\mbox{where}\ \Gamma=(\Gamma_{1},\dots,\Gamma_{11}).$$
	
	We denote the nonlinear terms by
	$$Q_P(u,v) = P^{\gamma\alpha\beta}\partial_\gamma u\partial_\alpha\partial_\beta v,\quad Q_0(u,v) = \partial^\alpha u\partial_\alpha v.$$

	We also introduce the good derivatives
	$$ G_a=\frac{1}{r}(x_a\partial_t+r\partial_a)=\omega_{a}\partial_t+\partial_a,\quad \mbox{for}\ a=1,2,3.$$
	For future notational convenience, we denote
	\begin{equation*}
		G=(G_1,G_2,G_3).
	\end{equation*}
	 
	The Fourier transform of $f \in L^2(\R^3)$ is defined by
	$$ \mathcal{F}(f)(\xi)=\hat{f}(\xi)=\int_{\R^{3}}f(x)e^{-ix\cdot\xi}\d x.$$

	\subsection{Basic estimates}
	
	In this subsection, we recall several preliminary estimates related to the vector fields and the null forms. We also collect some standard analytic tools, including the Klainerman–Sobolev inequality, H\"{o}lder inequality, and Hardy inequality, which will be frequently used in the subsequent analysis.
	
	\begin{lemma}[\cite{Sogge}]\label{lem:commutators} For any smooth function $u=u(t,x)$, we have
		\begin{equation}\label{est:commutators}
			\sum_{\alpha=0}^{3}\left|[\partial_\alpha,\Gamma^I]u\right|\lesssim\sum_{|J|<|I|}\sum_{\beta=0}^{3}|\partial_\beta\Gamma^Ju|.
		\end{equation}
	\end{lemma}

	\begin{lemma}[\cite{Alinhac2009,Sogge}]\label{lem:extra}
		For any smooth function $u=u(t,x)$, we have
		\begin{equation}
			\langle t-r\rangle|\partial u|+\langle t+r\rangle|G_a u|\lesssim\sum_{|J|=1}|\Gamma^J u|,\quad\forall\ a=1,2,3.\label{est:ope}
		\end{equation}
	\end{lemma}
	
	\begin{proof}
		By an elementary computation, for $a=1,2,3$, one obtains
		\begin{equation*}
			\begin{aligned}
				(t^2-r^2)\partial_t u&=tL_0 u-x^aL_a u,\\
				(t^2-r^2)\partial_1u
				&=tL_1u-x_1L_0u+x_2\Omega_{12}u+x_3\Omega_{13}u,\\
				(t^2-r^2)\partial_2u
				&=tL_2u-x_2L_0u-x_1\Omega_{12}u+x_3\Omega_{23}u,\\
				(t^2-r^2)\partial_3u
				&=tL_3u-x_3L_0u-x_1\Omega_{13}u-x_2\Omega_{23}u,\\
				G_au&=\frac{1}{r}(L_a u+(r-t)\partial_a u)=\frac{1}{t}(L_a u-\frac{x_a}{r}(r-t)\partial_t u).
			\end{aligned}
		\end{equation*}
		Combining the above identities, we arrive at~\eqref{est:ope}.
	\end{proof}
	
	Lemma \ref{lem:commutators} provides the standard commutator estimates for vector fields, while Lemma \ref{lem:extra} gives the so-called "good derivative" control. We next recall the estimates related to the null forms.

	\begin{lemma}[\cite{Sogge}]\label{lem:P&Q}
		For any $I\in \mathbb{N}^{11},$ and any smooth functions $u(t,x),v(t,x)$, there hold
		
		\begin{align}
			&|Q_P(u,v)|\lesssim|Gu||\partial\partial v|+|\partial u||G\partial v|,\label{est:P1}\\
			&|Q_0(u,v)|\lesssim|Gu||\partial v|+|\partial u||Gv|,\label{est:Q_0}\\
			&|\Gamma^IQ_P(u,v)|\lesssim\sum_{|I_1|+|I_2|\leq|I|}|Q_P(\Gamma^{I_1}u,\Gamma^{I_2}v)|,\label{est:P2}\\
	        &|\Gamma^IQ_0(u,v)|\lesssim\sum_{|I_1|+|I_2|\leq|I|}|Q_0(\Gamma^{I_1}u,\Gamma^{I_2}v)|.\label{est:Q0_1}
		\end{align}
	\end{lemma}
	
	\begin{proof}
		First, we observe that
		\begin{equation*}
			\begin{aligned}
				Q_P(u,v)=&P^{\gamma\alpha\beta}\partial_\gamma u\partial_\alpha\partial_\beta v\\
				=&P^{000}\partial_t u\partial_t\partial_t v+2P^{0a0}\partial_t u\partial_a\partial_t v+P^{c00}\partial_c u\partial_t\partial_t v\\
				&+P^{0ab}\partial_t u\partial_a\partial_b v+2P^{ca0}\partial_c u\partial_a\partial_t v+P^{cab}\partial_c u\partial_a\partial_b v,
			\end{aligned}
		\end{equation*}
		due to $P^{\gamma\alpha\beta}=P^{\gamma\beta\alpha}.$
		
		Based on the identity $\partial_a=G_a-\omega_{a}\partial_t$, we have
		\begin{equation*}
			\begin{aligned}
				P^{0a0}\partial_t u\partial_a\partial_t v=&P^{0a0}\partial_t uG_a \partial_t v-P^{0a0}\omega_{a}\partial_t u\partial_t \partial_t v,\\
				P^{c00}\partial_c u\partial_t\partial_t v=&P^{c00}G_c u\partial_t\partial_t v-P^{c00}\omega_{c}\partial_t u\partial_t\partial_t v,\\
				P^{0ab}\partial_t u\partial_a\partial_b v=&P^{0ab}\partial_t u G_a\partial_b v-P^{0ab}\omega_{a}\partial_t u\partial_tG_bv+P^{0ab}\omega_{a}\omega_{b}\partial_tu\partial_t\partial_tv,\\
				P^{ca0}\partial_c u\partial_a\partial_t v=&P^{ca0}G_c u\partial_a\partial_t v-P^{ca0}\omega_{c}\partial_t uG_a\partial_t v+P^{ca0}\omega_{c}\omega_{a}\partial_t u\partial_t\partial_t v,\\
				P^{cab}\partial_c u\partial_a\partial_b v=&P^{cab}G_c uG_a\partial_b v-P^{cab}\omega_{a}G_c u\partial_t\partial_b v-P^{cab}\omega_{c}\partial_t uG_a\partial_b v\\
				&+P^{cab}\omega_{c}\omega_{a}\partial_t u\partial_tG_b v-P^{cab}\omega_{c}\omega_{a}\omega_{b}\partial_t u\partial_t\partial_t v.
			\end{aligned}
		\end{equation*}
		The null condition leads us to
		\begin{equation*}
			\begin{aligned}
				&P^{000}\partial_t u\partial_t\partial_t v-2P^{0a0}\omega_{a}\partial_t u\partial_t \partial_t v-P^{c00}\omega_{c}\partial_t u\partial_t\partial_t v\\
				&+P^{0ab}\omega_{a}\omega_{b}\partial_tu\partial_t\partial_tv+2P^{ca0}\omega_{c}\omega_{a}\partial_t u\partial_t\partial_t v-P^{cab}\omega_{c}\omega_{a}\omega_{b}\partial_t u\partial_t\partial_t v\\
				=&P^{\gamma\alpha\beta}\xi_\gamma\xi_\alpha\xi_\beta\partial_tu\partial_t\partial_tv=0,
			\end{aligned}
		\end{equation*}
		where $\xi=(1,-\omega_1,-\omega_2,-\omega_3)$.
		
		Then based on $|G_au|\lesssim |\partial u|, [\partial_t,G_a]v=0$, we deduce
		\begin{equation*}
				|Q_P(u,v)|\lesssim|\partial u||G\partial v|+|Gu||\partial\partial v|.
		\end{equation*}
		This completes the proof of (\ref{est:P1}).
		
		Since $G_a=\omega_{a}\partial_t+\partial_a$, we obtain
		\begin{equation*}
			\begin{aligned}
				&Q_0(u,v)=-\partial_tu\partial_tv+\sum_a\partial_au\partial_av\\
				&=-\partial_tu\partial_tv+\sum_a(G_a-\omega_a\partial_t)u(G_a-\omega_a\partial_t)v\\
				&=\sum_a(G_auG_av-\omega_aG_au\partial_tv-\omega_a\partial_tuG_av).
			\end{aligned}
		\end{equation*}
		This yields estimate~\eqref{est:Q_0}.
		
		Furthermore, by applying the Leibniz rule together with an induction argument, we derive~(\ref{est:P2}) and~\eqref{est:Q0_1}.
	\end{proof}

    \begin{lemma}[\cite{Sogge}]\label{lem:P1}
    	For any smooth functions $f(t,x)$, $g(t,x)$, and $h(t,x)$, we have
    	\begin{align*}
    		&|P^{\gamma\alpha\beta}\partial_\alpha\partial_\gamma f\partial_\beta g|\lesssim|G\partial f||\partial g|+|\partial\partial f||Gg|;\\
    		&|P^{\gamma\alpha\beta}\partial_\gamma f\partial_\alpha g\partial_\beta h|\lesssim|Gf||\partial g||\partial h|+|\partial f||Gg||\partial h|+|\partial f||\partial g||Gh|.
    	\end{align*}
    \end{lemma}
    \begin{proof}
    	By applying $\partial_a=G_a-\omega_{a}\partial_t$ and using the null condition, we thus obtain the desired results.
    \end{proof}

    A key tool in establishing pointwise decay estimates is the Klainerman-Sobolev inequality~\cite{Sogge}, stated below.
	\begin{lemma}[Klainerman-Sobolev inequality]
		Let $u=u(t,x)$ be a sufficiently regular function that vanishes when $|x|$ is large. Then there holds
		\begin{equation}\label{est:Sobo}
			\langle t+r\rangle \langle t-r\rangle^{\frac{1}{2}}|u(t,x)| \lesssim \sum_{|I|\leq2}\|\Gamma^I u\|.
		\end{equation}
	\end{lemma}
	\begin{proof}
		For a proof of the Klainerman–Sobolev inequality, see Theorem 1.3 in \cite{Sogge}.
	\end{proof}

     For later use, we recall two standard inequalities: the H\"{o}lder inequality and the Hardy inequality.
     \begin{proposition}\label{prop:holder}
     	Let $0 <\lambda < 1$. Then
     	\begin{align*}
     		\||f|^{\lambda} |g|^{1-\lambda}\| \leq \|f\|^{\lambda}\|g\|^{1-\lambda}.
     	\end{align*}
     \end{proposition}
     
     \begin{proposition} \label{prop:hardy}
     	If $f\in H^1(\R^3)$ is smooth, then
     	\begin{align*}
     		\bigg \| \frac{f(x)}{|x|} \bigg\| \lesssim\|\partial_r f\|.
     	\end{align*} 
     \end{proposition}

     \subsection{Estimates on 3D wave equations}
	Various classical and recent results on wave equations can be found in \cite{Alinhac2009,Katayama2017,LiZhou,Sogge}. First, we recall several energy estimates for the 3D wave equation that will be used in later sections.
	\begin{lemma}\label{lem:energy}
		Let $u=u(t,x)$ be the solution to the Cauchy problem
		$$\left\{\begin{aligned}
			-\Box u(t,x)&=f(t,x),\\
			(u,\partial_t u)|_{t=0}&=(u_0,u_1),
		\end{aligned}\right.$$
		
		with $f(t,x)$ a sufficiently regular function. Then the following estimates hold.
		\begin{enumerate}
			
			\item\cite{Alinhac2001} {\rm {Ghost weight energy estimate:}}
			\begin{equation}\label{est:ghost}
				\begin{aligned}
					\mathcal{E}_{gst}^{\frac{1}{2}}(t,u)\lesssim \mathcal{E}_{gst}^\frac{1}{2}(0,u)+\int_{0}^{t}\|f(s,x)\| \d s,
				\end{aligned}
			\end{equation}
		    where $$\mathcal{E}_{gst}(t,u)=\int_{\R^{3}}|\partial u|^2\d x+\sum_{a=1}^{3}\int_{0}^{t}\int_{\R^{3}}\frac{|G_au|^2}{\langle r-s\rangle^{1+2\delta}}\d x\d s,\quad0<\delta\ll1.$$
			
			\item\cite{Alinhac2009} {\rm{Conformal energy estimate:}}
			\begin{equation}\label{est:Conformal}
				\begin{aligned}
					\mathcal{E}_{con}^{\frac{1}{2}}(t,u)\lesssim \mathcal{E}_{con}^\frac{1}{2}(0,u)+\int_{0}^{t}\|\langle s+r \rangle f(s,x)\|\d s,
				\end{aligned}
			\end{equation}
			where $$\mathcal{E}_{con}(t,u)=\int_{\R^{3}}((L_0 u)^2+u^2+\sum_{1\leq a<b\leq 3}|\Omega_{ab} u|^2+\sum_{a=1}^{3}|L_a u|^2)(t,x)\d x.$$
		\end{enumerate}
		
	\end{lemma}
	
	\begin{proof}
		\textbf{Proof of (i).} Let $q(r,t)=\int_{-\infty}^{r-t}\langle s\rangle^{-1-2\delta}\d s$ and set $e^q=e^{q(r,t)} $. Multiplying both sides of the equation by $e^q\partial_t u$, we have
		$$ \frac{1}{2}\partial_t(e^q|\partial u|^2)+\frac{1}{2}e^q\sum_{a}\frac{|G_au|^2}{\langle r-t\rangle^{1+2\delta}}-\partial^a (e^q\partial_a u\partial_t u)=f(t,x)(e^q\partial_t u).$$
		Note that $e^q\sim1$. Integrating over $\R^3$ in $x$ and from $0$ to $t$ in time, we get
		\begin{equation*}
			\begin{aligned}
				\mathcal{E}_{gst}(t,u)&\lesssim\mathcal{E}_{gst}(0,u)+\int_{0}^{t}\int_{\R^{3}}|f(s,x)(\partial_t u)|\d x\d s\\
				&\lesssim \mathcal{E}_{gst}(0,u)+\int_{0}^{t}\|f(s,x)\|\mathcal{E}_{gst}^\frac{1}{2}(s,u)\d s.
			\end{aligned}
		\end{equation*}
		Applying Gronwall inequality to the above, we obtain~\eqref{est:ghost}.
		
		\textbf{Proof of (ii).} Consider the nonspacelike multiplier
		\begin{equation*}
			\mathcal{M}_0=(r^2+t^2)\partial_t+2rt\partial_r,\quad\mbox{where}\ \partial_r=\omega^a\partial_a.
		\end{equation*}
		Multiplying  both sides of the equation by $\mathcal{M}_0 u+2tu$, we have
		\begin{equation*}
			\begin{aligned}
				&(-\Box u)(\mathcal{M}_0u+2tu)\\
				=&\frac{1}{2}\partial_t[(r^2+t^2)((\partial_tu)^2+\partial^bu\partial_bu)+4tr\partial_ru\partial_tu+4tu\partial_tu-2u^2+2\partial_a(x^au^2)]\\
				-&\partial^a[(r^2+t^2)\partial_au\partial_tu+tx_a((\partial_tu)^2-\partial^bu\partial_bu)+2tr\partial_ru\partial_au+2tu\partial_au+\partial_t(x_au^2)].
			\end{aligned}
		\end{equation*} 
		Hence integrating over $\R^3$ in $x$, we get
		\begin{equation}\label{equ:conformal}
			\begin{aligned}
				&\frac{1}{2}\partial_t\int_{\R^{3}}[(r^2+t^2)(\partial u)^2+4tr\partial_ru\partial_tu+4tu\partial_tu-2u^2+2\partial_a(x^au^2)]\d x\\
				=&\int_{\R^{3}}f(t,x)(\mathcal{M}_0u+2tu)\d x.
			\end{aligned}
		\end{equation}
		Let
		\begin{equation*}
			\begin{aligned}
				\tilde{E}(t,u)=\int_{\R^{3}}[(r^2+t^2)(\partial u)^2+4tr\partial_ru\partial_tu+4tu\partial_tu-2u^2+2\partial_a(x^au^2)](t,x)\d x.
			\end{aligned}
		\end{equation*}
	
		A straightforward calculation shows that
		\begin{equation*}
			\begin{aligned}
				\tilde{E}(t,u)=\int_{\R^{3}}[(L_0u+2u)^2+\sum_{1\leq a<b\leq3}|\Omega_{ab}u|^2+\sum_{a=1}^{3}|L_au|^2](t,x)\d x.
			\end{aligned}
		\end{equation*}
		By \eqref{equ:conformal}, we have
		\begin{equation*}
			\tilde{E}(t,u)\lesssim\tilde{E}(0,u)+\int_{0}^{t}\int_{\R^{3}}|f(s,x)(\mathcal{M}_0u+2su)|\d x\d s.
		\end{equation*}
		Recall the definition of $\mathcal{M}_0$, it holds
		\begin{equation*}
			\begin{aligned}
				\mathcal{M}_0u+2tu=&(r^2+t^2)\partial_tu+2rt\partial_ru+2tu\\
				=&t(L_0u+2u)+x^aL_au.
			\end{aligned}
		\end{equation*}
		Then
		\begin{equation}\label{equ:K_0}
			\tilde{E}(t,u)\lesssim\tilde{E}(0,u)+\int_{0}^{t}\|\langle s+r\rangle f(s,x)\|\tilde{E}^{\frac{1}{2}}(s,u)\d s.
		\end{equation}
		
		Applying Gronwall inequality to \eqref{equ:K_0}, we deduce
		\begin{equation}
			\tilde{E}^{\frac{1}{2}}(t,u)\lesssim \tilde{E}^{\frac{1}{2}}(0,u)+\int_{0}^{t}\|\langle s+r\rangle f(s,x)\|\d s.\label{equ:E2}
		\end{equation}
		It can be verified by careful calculations that
		\begin{equation}\label{equ:relation}
			\begin{aligned}
				\mathcal{E}_{con}(t,u)
				\sim&\tilde{E}(t,u).
			\end{aligned}
		\end{equation}
		Together with \eqref{equ:relation}, \eqref{equ:E2} yields the conformal energy estimate \eqref{est:Conformal}.
	\end{proof}

	To describe the asymptotic behavior of solutions to the inhomogeneous wave equation, we recall the scattering result from \cite{Sogge,Katayama2017}.
	\begin{lemma}\label{lem:scatterWave}
		Suppose that $u(t,x)$ is the solution to the Cauchy problem
		$$\left\{\begin{aligned}
			-\Box u(t,x) &=f(t,x),\\
			(u,\partial_t u)|_{t=0}&=(u_0,u_1).
		\end{aligned}\right.$$
		Suppose that 
		\begin{equation*}
			\int_{0}^{+\infty}\|f(\tau,x)\|_{\dot{H}^{s-1}}\d \tau<+\infty,
		\end{equation*}
		then there exist a pair $(\tilde{u}_0,\tilde{u}_1)\in \dot{H}^s(\mathbb{R}^3)\times \dot{H}^{s-1}(\mathbb{R}^3)$ and a constant $C$ such that
		\begin{equation*}
			\|(u,\partial_t u)^T-S_{W}(t)(\tilde{u}_0,\tilde{u}_1)^T\|_{\dot{H}^s\times \dot{H}^{s-1}}\leq C\int_{t}^{+\infty}\|f(\tau,x)\|_{\dot{H}^{s-1}}\d \tau.
		\end{equation*}
		In particular,
		\begin{equation*}
			\lim_{t\rightarrow+\infty}\|(u,\partial_t u)^T-S_{W}(t)(\tilde{u}_0,\tilde{u}_1)^T\|_{\dot{H}^s\times \dot{H}^{s-1}}=0.
		\end{equation*}
		In the above, $S_{W}(t)$ is denoted by
		\begin{equation*}
			\begin{pmatrix}
				\cos(t|\nabla|) & |\nabla|^{-1}\sin(t|\nabla|)\\
				-|\nabla|\sin(t|\nabla|) & \cos(t|\nabla|)
			\end{pmatrix}.
		\end{equation*}
		Here the operators are defined via the Fourier transform as
		\begin{equation*}
			\begin{aligned}
				\left(|\nabla|^{-1}f\right)(x)&=\mathcal{F}^{-1}\left(|\xi|^{-1}\hat{f}(\xi)\right)(x),\\
				\left(\cos(t|\nabla|)f\right)(x)&=\mathcal{F}^{-1}\left(\cos(t|\xi|)\hat{f}(\xi)\right)(x),
			\end{aligned}
		\end{equation*}
		and other operators are defined similarly.
		
		In addition, $S_{W}(t)(\tilde{u}_0,\tilde{u}_1)^T$ solves
		\begin{equation*}
			-\Box u=0, \quad(u,\partial_t u)|_{t=0}=(\tilde{u}_0,\tilde{u}_1).
		\end{equation*}
	\end{lemma}

	As the final result of this subsection, we establish the following lemma, which provides an $L^2$ estimate for the solution to the 3D wave equation.
	\begin{lemma}\label{lem:L2}
		Let $u=u(t,x)$ be the solution to the Cauchy problem
		$$\left\{\begin{aligned}
			-\Box u(t,x)&=f(t,x),\\
			(u,\partial_t u)|_{t=0}&=(u_0,u_1),
		\end{aligned}\right.$$
		
		with $f(t,x)$ a sufficiently regular function. Then for any $\eta >0$, we have
		\begin{equation}
			\begin{aligned}
				\|u(t,x)\|\lesssim \|u_0\|+\|u_1\|_{L^1\cap L^2}+\int_{0}^{t}\left(\langle s\rangle^{-\frac{\eta}{2}}\|f(s,x)\|_{L^1}+\langle s\rangle^{\eta}\|f(s,x)\|\right)\d s.\label{est:uL2}
			\end{aligned}
		\end{equation}
	\end{lemma}
	\begin{proof}
		Taking the Fourier transform of the Cauchy problem with respect to $x$ and applying Duhamel's principle, we get
		\begin{equation*}
			\left\{\begin{aligned}
				\partial_{tt} \hat{u}(t,\xi)+|\xi|^2\hat{u}(t,\xi)=\hat{f}(t,\xi),\\
				\hat{u}(0,\xi)=\hat{u}_0(\xi),\quad \partial_t\hat{u}(0,\xi)=\hat{u}_1(\xi).
			\end{aligned}\right.
		\end{equation*}

		Solving this second-order ODE in $t$, we get the solution in Fourier space:
		$$\hat{u}(t,\xi)=\cos(t|\xi|)\hat{u}_0(\xi)+\frac{\sin(t|\xi|)}{|\xi|}\hat{u}_1(\xi)+\int_{0}^{t}\frac{\sin((t-s)|\xi|)}{|\xi|}\hat{f}(s,\xi)\d s, $$
		for $(t,\xi)\in [0,\infty)\times\mathbb{R}^3$.
		
		Using the Plancherel theorem and a polar coordinate transformation, we deduce
		\begin{equation*}
			\begin{aligned}
				\|\cos(t|\xi|)\hat{u}_0(\xi)\|_{L^2_\xi}&\leq \|\hat{u}_0(\xi)\|_{L^2_\xi} \lesssim \|u_0\|,\\
				\left\|\frac{\sin(t|\xi|)}{|\xi|}\hat{u}_1(\xi)\right\|^2_{L^2_\xi}&=\int_{|\xi|\leq1} \frac{\sin^2(t|\xi|)}{|\xi|^2}|\hat{u}_1(\xi)|^2\d \xi+\int_{|\xi|\geq1} \frac{\sin^2(t|\xi|)}{|\xi|^2}|\hat{u}_1(\xi)|^2\d \xi\\
				&\leq\|\hat{u}_1(\xi)\|^2_{L^2_\xi}+\int_{|\omega|=1} \int_{0}^{1}\sin^2(tr)|\hat{u}_1(r)|^2\d r \d \omega\\
				&\lesssim \|\hat{u}_1(\xi)\|^2_{L^2_\xi}+\|\hat{u}_1(\xi)\|^2_{L^\infty_\xi}\\
				&\lesssim \|u_1\|^2+\|u_1\|^2_{L^1},\\
			\end{aligned}
		\end{equation*}
		\begin{equation*}
			\begin{aligned}
				&\left\|\frac{\sin((t-s)|\xi|)}{|\xi|}\hat{f}(s,\xi)\right\|^2_{L^2_\xi}\\
				=&\int_{|\xi|\leq \langle s\rangle^{-\eta}} \frac{\sin^2((t-s)|\xi|)}{|\xi|^2}|\hat{f}(s,\xi)|^2\d \xi
				+\int_{|\xi|\geq \langle s\rangle^{-\eta}} \frac{\sin^2((t-s)|\xi|)}{|\xi|^2}|\hat{f}(s,\xi)|^2\d \xi\\
				\lesssim& \langle s\rangle^{-\eta}\|\hat{f}(s,\xi)\|^2_{L^\infty_\xi}+\langle s\rangle^{2\eta}\|\hat{f}(s,\xi)\|^2_{L^2_\xi}\\
				\lesssim& \langle s\rangle^{-\eta}\|f(s,x)\|^2_{L^1}+\langle s\rangle^{2\eta}\|f(s,x)\|^2,
			\end{aligned}
		\end{equation*}
		for any $\eta>0.$
		
		Therefore, using again the Plancherel theorem, we obtain
		\begin{equation*}
			\begin{aligned}
				\|u(t,x)\|&\lesssim\|\hat{u}(t,\xi)\|_{L^2_\xi}\\
				&\lesssim\|u_0\|+\|u_1\|_{L^1}+\|u_1\|+\int_{0}^{t}\left(\langle s\rangle^{-\frac{\eta}{2}}\|f(s,x)\|_{L^1}+\langle s\rangle^\eta\|f(s,x)\|\right)\d s,
			\end{aligned}
		\end{equation*}
		which implies~\eqref{est:uL2}.
	\end{proof}

	\section{Global existence and scattering}\label{Se:exist}
	In this section, we establish global existence, uniform boundedness of the ghost weight energies, and scattering of the solution $w$ to~\eqref{equ:Wave}, under the conditions~\eqref{est:initial1} and~\eqref{est:initial2} in Theorem~\ref{thm:main1}. The proof is carried out via a bootstrap argument on the ghost weight energies of the solution.
	
	\subsection{Bootstrap setting}\label{Se:Boot}
	Let $N\in \mathbb{N}$ with $N\geq 5$. To prove Theorem~\ref{thm:main1}, we impose the following bootstrap assumptions on $w$: for $C_1>1$ and $0<\varepsilon\ll C_1^{-1}$ to be chosen later, 
	\begin{equation}\label{est:Bootw}
		\left\{\begin{aligned}
			\mathcal{E}_{gst}^{\frac{1}{2}}(t,\Gamma^I w)&\leq C_1,\quad\quad|I|\leq N,\\
			\mathcal{E}_{gst}^{\frac{1}{2}}(t,\partial\Gamma^I w)&\leq C_1\varepsilon,\quad\ \ |I|\leq N-1.\\
		\end{aligned}\right.
	\end{equation}

	For all initial data $(w_{0},w_{1})$ satisfying~\eqref{est:Bootw}, define
	\begin{equation*}
		T^{*}(w_{0},w_{1}):=\sup\left\{ t\in [0,\infty): w\ \mbox{satisfies}~\eqref{est:Bootw}\ \mbox{on}\ [0,t]\right\}.
	\end{equation*} 
	The following proposition is a key component of Theorem~\ref{thm:main1}.
	
	\begin{proposition}\label{pro:main}
		For all initial data $(w_{0},w_{1})$ satisfying the conditions~\eqref{est:initial1} and~\eqref{est:initial2} in Theorem~\ref{thm:main1}, it holds that $T^{*}(w_{0},w_{1})=+\infty$.
	\end{proposition}
	
	In the remainder of this section, we first focus on the proof of Proposition~\ref{pro:main}. From now on, the implied constants in $\lesssim$ are independent of the constants $C_1$ and $\varepsilon$ appearing in the bootstrap assumptions~\eqref{est:Bootw}.
	
	Note that, by the Klainerman-Sobolev inequality~\eqref{est:Sobo}, estimate \eqref{est:commutators}, and the bootstrap assumptions \eqref{est:Bootw}, for all $t\in [0,T^*(w_0,w_1))$, we have the following estimates for $w$ (with $0<\delta<\frac{1}{12}$)

    \begin{align}\label{est:L2norm1}
    	&\|\partial\Gamma^Iw\|+\left(\int_{0}^{t}\int_{\R^{3}}\frac{|G\Gamma^Iw|^2}{\langle r-s\rangle^{1+2\delta}}\d x\d s\right)^{\frac{1}{2}}\lesssim C_1,\quad\quad\quad|I|\leq N,\\
    	&\|\partial\partial\Gamma^Iw\|+\left(\int_{0}^{t}\int_{\R^{3}}\frac{|G\partial\Gamma^Iw|^2}{\langle r-s\rangle^{1+2\delta}}\d x\d s\right)^{\frac{1}{2}}\lesssim C_1\varepsilon,\quad\quad|I|\leq N-1,\label{est:L2norm2}\\
    	&|\partial \Gamma^I w|\l\lesssim C_1\langle t+r\rangle^{-1}\langle t-r\rangle^{-\frac{1}{2}},\quad\quad\quad\quad\quad\quad\quad\quad\quad\ |I|\leq N-2,\label{est:pointwise1}\\
    	&|\partial\partial \Gamma^I w|\l\lesssim C_1 \varepsilon \langle t+r\rangle^{-1}\langle t-r\rangle^{-\frac{1}{2}},\quad\quad\quad\quad\quad\quad\quad\quad\ |I|\leq N-3.\label{est:pointwise2}
    \end{align}

	\subsection{Estimates under the bootstrap assumptions}
	In this subsection, we first derive the pointwise estimates for the solution $w$ under the bootstrap assumptions~\eqref{est:Bootw}. These estimates will serve as a key tool in controlling the nonlinear terms and in closing the bootstrap argument.
	
	\begin{lemma}
		Let $N\in \mathbb{N}$ with $N\geq5$. For all $t\in[0,T^*(w_0,w_1))$, the following estimates hold:
		
		\begin{enumerate}
			
			\item {\rm {Pointwise estimates on $\Gamma^Iw$}}. For $|I|\leq N-2$, we have
			\begin{equation}\label{est:gamma}
				|\Gamma^Iw|\lesssim C_1\langle t+r\rangle^{-\frac{1}{2}+\delta}.
			\end{equation}
			
			\item {\rm {Pointwise estimates on $\partial\partial\Gamma^Iw$}}. For $|I|\leq N-3$, we have
			\begin{equation}\label{est:partialdouble}
				|\partial\partial\Gamma^Iw|\lesssim C_1\varepsilon^{\frac{1}{2}}\langle t+r\rangle^{-1}\langle t-r\rangle^{-1}.
			\end{equation}
			
			\item {\rm {Pointwise estimates on $\partial\Gamma^Iw$}}. For $|I|\leq N-3$, we have
			\begin{equation}\label{est:partial}
				|\partial\Gamma^Iw|\lesssim \min\{C_1\varepsilon\langle t+r\rangle^{-\frac{1}{2}+\delta}, C_1\varepsilon^{\frac{1}{2}}\langle t+r\rangle^{-1+\delta}, C_1\langle t+r\rangle^{-\frac{1}{2}+\delta}\langle t-r\rangle^{-1}\}.
			\end{equation}
			
		\end{enumerate}
		
	\end{lemma}

	\begin{proof}
		\textbf{Proof of (i).} Suppose that $|I|\leq N-2$ with $N\geq5$. By estimate~\eqref{est:pointwise1} and the fundamental theorem of calculus, integrating $\partial_r\Gamma^Iw(t,x=r\omega)$ from $r$ to $+\infty$ with $\omega$ fixed gives
		\begin{equation*}
			\begin{aligned}
				|\Gamma^Iw(t,x)|&\leq\left|\int_{r}^{+\infty}\partial_\rho\Gamma^Iw(t,\rho\omega)\d \rho\right|\leq\int_{r}^{+\infty}\left|\partial_\rho\Gamma^Iw(t,\rho\omega)\right|\d \rho\\
				&\lesssim\int_{r}^{+\infty}C_1\langle t+\rho\rangle^{-1}\langle t-\rho\rangle^{-\frac{1}{2}}\d \rho\\
				&\lesssim\int_{r}^{+\infty}C_1\langle t+\rho\rangle^{-\frac{1}{2}+\delta}\langle t-\rho\rangle^{-1-\delta}\d \rho\\
				&\lesssim C_1\langle t+r\rangle^{-\frac{1}{2}+\delta}.
			\end{aligned}
		\end{equation*}
		Here we have used the fact that $\Gamma^Iw(t,+\infty)=0$, and $\delta\in(0,\frac{1}{12})$ is a small parameter. Recall that $\partial_r=\omega^a\partial_a$.
		
		\textbf{Proof of (ii).}
		Let $|I|\leq N-3$. Applying Lemmas~\ref{lem:commutators},~\ref{lem:extra}, and estimate~\eqref{est:pointwise1}, we have
		\begin{equation*}
			|\partial\partial\Gamma^Iw|\lesssim\langle t-r\rangle^{-1}\sum_{|J|=1}|\Gamma^J\partial\Gamma^Iw|\lesssim C_1\langle t+r\rangle^{-1}\langle t-r\rangle^{-\frac{3}{2}}.
		\end{equation*}
	    
	    Interpolating the above estimate with the pointwise bound~\eqref{est:pointwise2} then yields
	    \begin{equation*}
	    	|\partial\partial\Gamma^Iw|\lesssim C_1\varepsilon^{\frac{1}{2}}\langle t+r\rangle^{-1}\langle t-r\rangle^{-1}, \quad|I|\leq N-3.
	    \end{equation*}
	    
		\textbf{Proof of (iii).} Similarly, for $|I|\leq N-3$, by estimate~\eqref{est:pointwise2} and the fundamental theorem of calculus, integrating $\partial_r\partial\Gamma^Iw(t,x=r\omega)$ from $r$ to $+\infty$ with $\omega$ fixed gives
		\begin{equation*}
			\begin{aligned}
				|\partial\Gamma^Iw(t,x)|&\leq\left|\int_{r}^{+\infty}\partial_\rho\partial\Gamma^Iw(t,\rho\omega)\d \rho\right|\leq\int_{r}^{+\infty}\left|\partial_\rho\partial\Gamma^Iw(t,\rho\omega)\right|\d \rho\\
				&\lesssim\int_{r}^{+\infty}C_1\varepsilon\langle t+\rho\rangle^{-1}\langle t-\rho\rangle^{-\frac{1}{2}}\d \rho\\
				&\lesssim C_1\varepsilon\langle t+r\rangle^{-\frac{1}{2}+\delta}.
			\end{aligned}
		\end{equation*}
		Here we have used the fact that $\partial\Gamma^Iw(t,+\infty)=0$, and $\delta\in(0,\frac{1}{12})$ is a small parameter. 
	    
	    Substituting estimate~\eqref{est:partialdouble} into the above inequality, we obtain
	    \begin{equation*}
	    	\begin{aligned}
	    		|\partial\Gamma^Iw(t,x)|&\leq\left|\int_{r}^{+\infty}\partial_\rho\partial\Gamma^Iw(t,\rho\omega)\d \rho\right|\leq\int_{r}^{+\infty}\left|\partial_\rho\partial\Gamma^Iw(t,\rho\omega)\right|\d \rho\\
	    		&\lesssim C_1\varepsilon^{\frac{1}{2}}\langle t+r\rangle^{-1+\delta}.
	    	\end{aligned}
	    \end{equation*}
	    
	    By using~\eqref{est:gamma} and Lemma~\ref{lem:extra}, we have
	    \begin{equation*}
	    	|\partial\Gamma^Iw|\lesssim\langle t-r\rangle^{-1}\sum_{|J|=1}|\Gamma^J\Gamma^Iw|\lesssim C_1\langle t+r\rangle^{-\frac{1}{2}+\delta}\langle t-r\rangle^{-1}.
	    \end{equation*}
	    This completes the proof.
	\end{proof}

	\begin{lemma}
		Let $N\in \mathbb{N}$ with $N\geq 5$. For all $t\in[0,T^*(w_0,w_1))$, and any $0\leq\lambda\leq1$, we have the following estimate:
		\begin{equation}\label{est:partialrefined}
			|\partial\Gamma^Iw|\lesssim C_1\varepsilon^{\frac{\lambda}{2}}\langle t+r\rangle^{-\frac{\lambda}{2}-\frac{1}{2}+\delta}\langle t-r\rangle^{-1+\lambda},\quad |I|\leq N-3.
		\end{equation}
	\end{lemma}
	
	\begin{proof}
		By interpolating the estimates in~\eqref{est:partial}, we obtain
		\begin{equation*}
			|\partial\Gamma^Iw|\lesssim C_1\varepsilon^{\frac{\lambda}{2}}\langle t+r\rangle^{-\frac{\lambda}{2}-\frac{1}{2}+\delta}\langle t-r\rangle^{-1+\lambda},
		\end{equation*}
	where $0\leq\lambda\leq1.$
	\end{proof}

	\begin{lemma}
		Let $N\in \mathbb{N}$ with $N\geq5$. For all $t\in[0,T^*(w_0,w_1))$, we have the following estimates:
		
		\begin{enumerate}
			\item {\rm {Pointwise estimates on $G\Gamma^Iw$}}. For $|I|\leq N-3$, we have
			\begin{equation}\label{est:Ggamma}
				|G\Gamma^Iw|\lesssim\min\{C_1\langle t+r\rangle^{-\frac{3}{2}+\delta}, C_1\varepsilon^{\frac{1}{4}}\langle t+r\rangle^{-\frac{5}{4}+\delta}\}.
			\end{equation}
		    
		    \item {\rm {Pointwise estimates on $G\partial\Gamma^Iw$}}. For $|I|\leq N-4$, we have
		    \begin{equation}\label{est:Gpartialgamma}
		    	|G\partial\Gamma^Iw|\lesssim \min\{C_1\varepsilon^{\frac{1}{2}}\langle t+r\rangle^{-2+\delta}, C_1\varepsilon\langle t+r\rangle^{-\frac{3}{2}+\delta}\}.
		    \end{equation}
		    
		    \item {\rm {Pointwise estimates on $G\partial\partial\Gamma^Iw$}}. For $|I|\leq N-4$, we have
		    \begin{equation}\label{est:Gpartial2gamma}
		    	|G\partial\partial\Gamma^Iw|\lesssim C_1\varepsilon\langle t+r\rangle^{-2}\langle t-r\rangle^{-\frac{1}{2}}.
		    \end{equation}
		    
		\end{enumerate}
	\end{lemma}

	\begin{proof}
		\textbf{Proof of (i).} First, for $|I|\leq N-3$, by estimate~\eqref{est:gamma} and Lemma~\ref{lem:extra}, we obtain
		\begin{equation*}
			|G\Gamma^Iw|\lesssim\langle t+r\rangle^{-1}\sum_{|J|=1}|\Gamma^J\Gamma^Iw|\lesssim C_1\langle t+r\rangle^{-\frac{3}{2}+\delta}.
		\end{equation*}
	    
	    By the definition of $G$ and \eqref{est:partial}, we deduce
	    \begin{equation*}
	    	|G\Gamma^Iw|\lesssim|\partial\Gamma^Iw|\lesssim C_1\varepsilon^{\frac{1}{2}}\langle t+r\rangle^{-1+\delta},\quad|I|\leq N-3.
	    \end{equation*}
	    
	    By interpolating the above estimates, we obtain
	    \begin{equation*}
	    	|G\Gamma^Iw|\lesssim C_1\varepsilon^{\frac{1}{4}}\langle t+r\rangle^{-\frac{5}{4}+\delta}.
	    \end{equation*}

	    \textbf{Proof of (ii).}
	    For $|I|\leq N-4$, Lemma~\ref{lem:extra} yields
	    \begin{equation*}
	    	|G\partial\Gamma^Iw|\lesssim\langle t+r\rangle^{-1}\sum_{|J|=1}|\Gamma^J\partial\Gamma^Iw|.
	    \end{equation*}
        Then, by Lemma~\ref{lem:commutators} and estimate~\eqref{est:partial}, we get the desired results.
    
        \textbf{Proof of (iii).}
        For $|I|\leq N-4$, by~\eqref{est:pointwise2}, we have
        \begin{equation*}
        	|G\partial\partial\Gamma^Iw|\lesssim\langle t+r\rangle^{-1}\sum_{|J|=1}|\Gamma^J\partial\partial\Gamma^Iw|\lesssim C_1\varepsilon\langle t+r\rangle^{-2}\langle t-r\rangle^{-\frac{1}{2}}.
        \end{equation*}
    Here we have used Lemmas~\ref{lem:commutators} and~\ref{lem:extra}.
    
    This completes the proof.
	\end{proof}

	In the following, we establish estimates for the nonlinear terms in the equation for $w$ under the bootstrap assumptions~\eqref{est:Bootw}. These estimates are crucial for the proof of Proposition~\ref{pro:main}.
	
	\begin{lemma}\label{est:L1L2nonlinear}
		Let $N\in\mathbb{N}$ with $N\geq5$. For all $t\in[0,T^*(w_0,w_1))$, the following estimates hold.
		
		\begin{enumerate}
			\item {\rm {$L^1_tL^2_x$ estimates on $\Gamma^IQ_P(w,w)$}}. For $|I|\leq N-1$, we have
			\begin{equation*}
				\int_{0}^{t}\|\Gamma^IQ_P(w,w)\|\d s\lesssim C_1^2\varepsilon.
			\end{equation*}
			
			\item {\rm {$L^1_tL^2_x$ estimates on $\Gamma^IQ_0(w,w)$}}. For $|I|\leq N-1$, we have
			\begin{equation*}
				\int_{0}^{t}\|\Gamma^IQ_0(w,w)\|\d s\lesssim C_1^2\varepsilon^{\frac{1}{4}}.
			\end{equation*}
		
		    \item {\rm {$L^1_tL^2_x$ estimates on $\partial\Gamma^IQ_P(w,w)$}}. For $|I|\leq N-2$, we have
		    \begin{equation*}
		    	\int_{0}^{t}\|\partial\Gamma^IQ_P(w,w)\|\d s\lesssim C_1^2\varepsilon^{\frac{5}{4}}.
		    \end{equation*}
		    
		    \item {\rm {$L^1_tL^2_x$ estimates on $\partial\Gamma^IQ_0(w,w)$}}. For $|I|\leq N-2$, we have
		    \begin{equation*}
		    	\int_{0}^{t}\|\partial\Gamma^IQ_0(w,w)\|\d s\lesssim C_1^2\varepsilon^{\frac{5}{4}}.
		    \end{equation*}
			
		\end{enumerate}
		
	\end{lemma}
	
	\begin{proof}
		\textbf{Proof of (i).} Let $|I|\leq N-1$. Applying Lemma~\ref{lem:P&Q}, we obtain
		\begin{equation*}
			\begin{aligned}
				&\|\Gamma^IQ_P(w,w)\|\lesssim\sum_{|I_1|+|I_2|\leq|I|}\|Q_P(\Gamma^{I_1}w,\Gamma^{I_2}w)\|\\
				&\lesssim\sum_{\substack{|I_1|+|I_2|\leq|I|}}\|G\Gamma^{I_1}w\partial\partial\Gamma^{I_2}w\|
				+\sum_{\substack{|I_1|+|I_2|\leq|I|}}\|\partial\Gamma^{I_1}wG\partial\Gamma^{I_2}w\|
				=\mathcal{I}_1+\mathcal{I}_2.
			\end{aligned}
		\end{equation*}
	
	\textbf{Bound for $\mathcal{I}_1.$} Estimates~\eqref{est:L2norm2},~\eqref{est:pointwise2},~\eqref{est:Ggamma}, and the H\"{o}lder inequality imply
	\begin{equation}\label{est:I_1}
		\begin{aligned}
			\mathcal{I}_1&\lesssim\sum_{\substack{|I_1|+|I_2|\leq|I|\\|I_2|\leq N-3}}\left\|\frac{G\Gamma^{I_1}w}{\langle s-r\rangle^{\frac{1}{2}+\delta}}\right\|\|\langle s-r\rangle^{\frac{1}{2}+\delta}\partial\partial\Gamma^{I_2}w\|_{L^\infty}\\
			&+\sum_{\substack{|I_1|+|I_2|\leq|I|\\|I_1|\leq N-3}}\|G\Gamma^{I_1}w\|_{L^\infty}\|\partial\partial\Gamma^{I_2}w\|\\
			&\lesssim C_1\varepsilon\langle s\rangle^{-1+\delta}\sum_{\substack{|I_1|\leq|I|}}\left\|\frac{G\Gamma^{I_1}w}{\langle s-r\rangle^{\frac{1}{2}+\delta}}\right\|+C_1^2\varepsilon^{\frac{5}{4}}\langle s\rangle^{-\frac{5}{4}+\delta}.
		\end{aligned}
	\end{equation}

    \textbf{Bound for $\mathcal{I}_2.$} Using the H\"{o}lder inequality again, we get
    \begin{equation}\label{est:I_2}
    	\begin{aligned}
    		\mathcal{I}_2&\lesssim\sum_{\substack{|I_1|+|I_2|\leq|I|\\|I_1|\leq N-2}}\|\langle s-r\rangle^{\frac{1}{2}+\delta}\partial\Gamma^{I_1}w\|_{L^\infty}\left\|\frac{G\partial\Gamma^{I_2}w}{\langle s-r\rangle^{\frac{1}{2}+\delta}}\right\|\\
    		&+\sum_{\substack{|I_1|+|I_2|\leq|I|\\|I_2|\leq N-4}}\|\partial\Gamma^{I_1}w\|\|G\partial\Gamma^{I_2}w\|_{L^\infty}\\
    		&\lesssim C_1^2\varepsilon\langle s\rangle^{-\frac{3}{2}+\delta}+C_1\langle s\rangle^{-1+\delta}\sum_{\substack{|I_2|\leq|I|}}\left\|\frac{G\partial\Gamma^{I_2}w}{\langle s-r\rangle^{\frac{1}{2}+\delta}}\right\|,
    	\end{aligned}
    \end{equation}
	    where we have used estimates~\eqref{est:L2norm1},~\eqref{est:pointwise1}, and~\eqref{est:Gpartialgamma} in the last inequality.
	
	    By the H\"{o}lder inequality,~\eqref{est:I_1},~\eqref{est:I_2}, \eqref{est:L2norm1}, and \eqref{est:L2norm2}, we obtain
	    \begin{equation*}
	    	\int_{0}^{t}\|\Gamma^IQ_P(w,w)\|\d s\lesssim C_1^2\varepsilon.
	    \end{equation*}
		
		\textbf{Proof of (ii).} Ler $|I|\leq N-1$. Combining Lemma~\ref{lem:P&Q} with the H\"{o}lder inequality, we obtain
		\begin{equation*}
			\begin{aligned}
				&\|\Gamma^IQ_0(w,w)\|\lesssim\sum_{\substack{|I_1|+|I_2|\leq|I|}}\|G\Gamma^{I_1}w\partial\Gamma^{I_2}w\|\\
				&\lesssim\sum_{\substack{|I_1|+|I_2|\leq|I|\\|I_2|\leq N-3}}\left\|\frac{G\Gamma^{I_1}w}{\langle s-r\rangle^{\frac{1}{2}+\delta}}\right\|\|\langle s-r\rangle^{\frac{1}{2}+\delta}\partial\Gamma^{I_2}w\|_{L^\infty}+\sum_{\substack{|I_1|+|I_2|\leq|I|\\|I_1|\leq N-3}}\|G\Gamma^{I_1}w\|_{L^\infty}\|\partial\Gamma^{I_2}w\|\\
				&\lesssim C_1\varepsilon^{\frac{1}{4}}\langle s\rangle^{-\frac{3}{4}+2\delta}\sum_{\substack{|I_1|\leq|I|}}\left\|\frac{G\Gamma^{I_1}w}{\langle s-r\rangle^{\frac{1}{2}+\delta}}\right\|+C_1^2\varepsilon^{\frac{1}{4}}\langle s\rangle^{-\frac{5}{4}+\delta}.
			\end{aligned}
		\end{equation*}
	    Here we have used estimates~\eqref{est:L2norm1},~\eqref{est:Ggamma}, and~\eqref{est:partialrefined} (with $\lambda=\frac{1}{2}$). By using the H\"{o}lder inequality and \eqref{est:L2norm1} again, we obtain
		\begin{equation*}
			\int_{0}^{t}\|\Gamma^IQ_0(w,w)\|\d s\lesssim C_1^2\varepsilon^{\frac{1}{4}}.
		\end{equation*}
	    
	    \textbf{Proof of (iii).} Let $|I|\leq N-2$. By Lemma~\ref{lem:P&Q}, we get
	    
	    \begin{equation}\label{est:II}
	    	\begin{aligned}
	    		&\|\partial\Gamma^IQ_P(w,w)\|\lesssim\sum_{\substack{|I_1|+|I_2|\leq|I|}}\left(\|Q_P(\partial\Gamma^{I_1}w,\Gamma^{I_2}w)\|+\|Q_P(\Gamma^{I_1}w,\partial\Gamma^{I_2}w)\|\right)\\
	    		&\lesssim\sum_{\substack{|I_1|+|I_2|\leq|I|}}\left(\|G\partial\Gamma^{I_1}w\partial\partial\Gamma^{I_2}w\|+\|\partial\Gamma^{I_1}wG\partial\partial\Gamma^{I_2}w\|+\|G\Gamma^{I_1}w\partial\partial\partial\Gamma^{I_2}w\|\right)\\
	    		&=\mathcal{II}_1+\mathcal{II}_2+\mathcal{II}_3.
	    	\end{aligned}
	    \end{equation}
	    
	    \textbf{Bound for $\mathcal{II}_1.$} Estimates~\eqref{est:L2norm2},~\eqref{est:pointwise2},~\eqref{est:Gpartialgamma}, and the H\"{o}lder inequality imply
	    \begin{equation*}
	    	\begin{aligned}
	    		\mathcal{II}_1&\lesssim\sum_{\substack{|I_1|+|I_2|\leq|I|\\|I_2|\leq N-3}}\left\|\frac{G\partial\Gamma^{I_1}w}{\langle s-r\rangle^{\frac{1}{2}+\delta}}\right\|\|\langle s-r\rangle^{\frac{1}{2}+\delta}\partial\partial\Gamma^{I_2}w\|_{L^\infty}\\
	    		&+\sum_{\substack{|I_1|+|I_2|\leq|I|\\|I_1|\leq N-4}}\|G\partial\Gamma^{I_1}w\|_{L^\infty}\|\partial\partial\Gamma^{I_2}w\|\\
	    		&\lesssim C_1\varepsilon\langle s\rangle^{-1+\delta}\sum_{\substack{|I_1|\leq|I|}}\left\|\frac{G\partial\Gamma^{I_1}w}{\langle s-r\rangle^{\frac{1}{2}+\delta}}\right\|+(C_1\varepsilon)^2\langle s\rangle^{-\frac{3}{2}+\delta}.
	    	\end{aligned}
	    \end{equation*}

        \textbf{Bound for $\mathcal{II}_2.$} Using the H\"{o}lder inequality again, we get
        \begin{equation*}
        	\begin{aligned}
        		\mathcal{II}_2&\lesssim\sum_{\substack{|I_1|+|I_2|\leq|I|\\|I_2|\leq N-4}}\left\|\frac{\partial\Gamma^{I_1}w}{\langle r\rangle^{\frac{1}{4}}}\right\|\|\langle r\rangle^{\frac{1}{4}}G\partial\partial\Gamma^{I_2}w\|_{L^\infty}\\
        		&+\sum_{\substack{|I_1|+|I_2|\leq|I|\\|I_1|\leq N-3,|J|\leq1}}\|\langle s+r\rangle^{-1}\partial\Gamma^{I_1}w\|_{L^\infty}\|\partial\partial\Gamma^J\Gamma^{I_2}w\|,
        	\end{aligned}
        \end{equation*}
        where we have used the following estimate
        \begin{equation*}
        	|G\partial\partial\Gamma^{I_2}w|\lesssim\langle s+r\rangle^{-1}\sum_{|J|=1}|\Gamma^J\partial\partial\Gamma^{I_2}w|.
        \end{equation*}
	    
	    Propositions~\ref{prop:holder} and~\ref{prop:hardy} give
	    \begin{equation*}
	    	\left\|\frac{\partial\Gamma^{I_1}w}{\langle r\rangle^{\frac{1}{4}}}\right\|\lesssim\|\partial\partial\Gamma^{I_1}w\|^{\frac{1}{4}}\|\partial\Gamma^{I_1}w\|^{\frac{3}{4}}.
	    \end{equation*}
	    
	    By estimates~\eqref{est:L2norm1},~\eqref{est:L2norm2},~\eqref{est:partial}, and~\eqref{est:Gpartial2gamma}, we obtain
	    \begin{equation*}
	    	\mathcal{II}_2\lesssim C_1^2\varepsilon^{\frac{5}{4}}\langle s\rangle^{-\frac{3}{2}+\delta}.
	    \end{equation*}
	    
	    \textbf{Bound for $\mathcal{II}_3.$} Using the H\"{o}lder inequality again, we get
	    \begin{equation}\label{est:II3}
	    	\begin{aligned}
	    		\mathcal{II}_3&\lesssim\sum_{\substack{|I_1|+|I_2|\leq|I|\\|I_2|\leq N-4}}\left\|\frac{G\Gamma^{I_1}w}{\langle s-r\rangle^{\frac{3}{4}(\frac{1}{2}+\delta)}\langle r\rangle^{\frac{1}{4}}}\right\|\|\langle s-r\rangle^{\frac{3}{4}(\frac{1}{2}+\delta)}\langle r\rangle^{\frac{1}{4}}\partial\partial\partial\Gamma^{I_2}w\|_{L^\infty}\\
	    		&+\sum_{\substack{|I_1|+|I_2|\leq|I|\\|I_1|\leq N-3}}\|G\Gamma^{I_1}w\|_{L^\infty}\|\partial\partial\partial\Gamma^{I_2}w\|.
	    	\end{aligned}
	    \end{equation}
    Propositions~\ref{prop:holder} and~\ref{prop:hardy} give the following estimate
    \begin{equation}\label{est:GGammaL2}
    	\begin{aligned}
    		&\left\|\frac{G\Gamma^{I_1}w}{\langle r-s\rangle^{\frac{3}{4}(\frac{1}{2}+\delta)}\langle r\rangle^{\frac{1}{4}}}\right\|\lesssim\left\|\frac{G\Gamma^{I_1}w}{\langle r\rangle}\right\|^{\frac{1}{4}}\left\|\frac{G\Gamma^{I_1}w}{\langle r-s\rangle^{\frac{1}{2}+\delta}}\right\|^{\frac{3}{4}}\\
    		&\lesssim\|\partial_rG\Gamma^{I_1}w\|^{\frac{1}{4}}\left\|\frac{G\Gamma^{I_1}w}{\langle r-s\rangle^{\frac{1}{2}+\delta}}\right\|^{\frac{3}{4}}
    		\lesssim\|\partial\partial\Gamma^{I_1}w\|^{\frac{1}{4}}\left\|\frac{G\Gamma^{I_1}w}{\langle r-s\rangle^{\frac{1}{2}+\delta}}\right\|^{\frac{3}{4}}.
    	\end{aligned}
    \end{equation}
    
    Substituting the above inequality~\eqref{est:GGammaL2} into~\eqref{est:II3}, we obtain
	\begin{equation*}
		\begin{aligned}
			\mathcal{II}_3
			&\lesssim(C_1\varepsilon)^{\frac{5}{4}}\langle s\rangle^{-\frac{3}{4}}\sum_{\substack{|I_1|\leq|I|}}\left\|\frac{G\Gamma^{I_1}w}{\langle s-r\rangle^{\frac{1}{2}+\delta}}\right\|^{\frac{3}{4}}+C_1^2\varepsilon^{\frac{5}{4}}\langle s\rangle^{-\frac{5}{4}+\delta}.
		\end{aligned}
	\end{equation*}
	Here we have used estimates~\eqref{est:L2norm2},~\eqref{est:pointwise2}, and \eqref{est:Ggamma}.

	    Substituting bounds for $\mathcal{II}_1, \mathcal{II}_2, \mathcal{II}_3$ into~\eqref{est:II}, we obtain
	    \begin{equation*}
	    	\int_{0}^{t}\|\partial\Gamma^IQ_n(w,w)\|\d s\lesssim C_1^2\varepsilon^{\frac{5}{4}}.
	    \end{equation*}
	    Here we have used the H\"{o}lder inequality, \eqref{est:L2norm1}, and \eqref{est:L2norm2}. 
	    
	    \textbf{Proof of (iv).} Ler $|I|\leq N-2$. By Lemma~\ref{lem:P&Q}, we have
	    \begin{equation*}
	    	\begin{aligned}
	    		&\|\partial\Gamma^IQ_0(w,w)\|\lesssim\sum_{\substack{|I_1|+|I_2|\leq|I|}}(\|Q_0(\partial\Gamma^{I_1}w,\Gamma^{I_2}w)\|+\|Q_0(\Gamma^{I_1}w,\partial\Gamma^{I_2}w)\|)\\
	    		&\lesssim\sum_{\substack{|I_1|+|I_2|\leq|I|}}(\|G\partial\Gamma^{I_1}w\partial\Gamma^{I_2}w\|+\|\partial\partial\Gamma^{I_1}wG\Gamma^{I_2}w\|)=\mathcal{III}_1+\mathcal{III}_2.
	    	\end{aligned}
	    \end{equation*}
        
        \textbf{Bound for $\mathcal{III}_1.$} It follows from the H\"{o}lder inequality that
        \begin{equation}\label{est:III1}
        	\begin{aligned}
        		\mathcal{III}_1&\lesssim\sum_{\substack{|I_1|+|I_2|\leq|I|\\|I_2|\leq N-3}}\left\|\frac{G\partial\Gamma^{I_1}w}{\langle s-r\rangle^{\frac{1}{2}+\delta}}\right\|\|\langle s-r\rangle^{\frac{1}{2}+\delta}\partial\Gamma^{I_2}w\|_{L^\infty}\\
        		&+\sum_{\substack{|I_1|+|I_2|\leq|I|\\|I_1|\leq N-4}}\|\langle r\rangle^{\frac{1}{4}}G\partial\Gamma^{I_1}w\|_{L^\infty}\left\|\frac{\partial\Gamma^{I_2}w}{\langle r\rangle^{\frac{1}{4}}}\right\|.
        	\end{aligned}
        \end{equation}
        
        By Propositions~\ref{prop:holder},~\ref{prop:hardy},~\eqref{est:L2norm1}, and~\eqref{est:L2norm2}, we obtain
        \begin{equation*}
        	\begin{aligned}
        		\left\|\frac{\partial\Gamma^{I_2}w}{\langle r\rangle^{\frac{1}{4}}}\right\|\lesssim\left\|\frac{\partial\Gamma^{I_2}w}{\langle r\rangle}\right\|^{\frac{1}{4}}\|\partial\Gamma^{I_2}w\|^{\frac{3}{4}}\lesssim\|\partial\partial\Gamma^{I_2}w\|^{\frac{1}{4}}\|\partial\Gamma^{I_2}w\|^{\frac{3}{4}}\lesssim C_1\varepsilon^{\frac{1}{4}}.
        	\end{aligned}
        \end{equation*}
        Substituting the above estimate into~\eqref{est:III1} and using \eqref{est:partialrefined} (with $\lambda=\frac{1}{2}$) together with~\eqref{est:Gpartialgamma}, we obtain
        \begin{equation*}
        	\begin{aligned}
        		\mathcal{III}_1\lesssim C_1\varepsilon^{\frac{1}{4}}\langle s\rangle^{-\frac{3}{4}+2\delta}\sum_{\substack{|I_1|\leq|I|}}\left\|\frac{G\partial\Gamma^{I_1}w}{\langle s-r\rangle^{\frac{1}{2}+\delta}}\right\|+C_1^2\varepsilon^{\frac{5}{4}}\langle s\rangle^{-\frac{5}{4}+\delta}.
        	\end{aligned}
        \end{equation*}
        
        \textbf{Bound for $\mathcal{III}_2.$} Using the H\"{o}lder inequality again, we get
        \begin{equation}\label{est:III2}
        	\begin{aligned}
        		\mathcal{III}_2&\lesssim\sum_{\substack{|I_1|+|I_2|\leq|I|\\|I_2|\leq N-3}}\|\partial\partial\Gamma^{I_1}w\|\|G\Gamma^{I_2}w\|_{L^\infty}\\
        		&+\sum_{\substack{|I_1|+|I_2|\leq|I|\\|I_1|\leq N-3}}\|\langle s-r\rangle^{\frac{3}{4}(\frac{1}{2}+\delta)}\langle r\rangle^{\frac{1}{4}}\partial\partial\Gamma^{I_1}w\|_{L^\infty}\left\|\frac{G\Gamma^{I_2}w}{\langle s-r\rangle^{\frac{3}{4}(\frac{1}{2}+\delta)}\langle r\rangle^{\frac{1}{4}}}\right\|.
        	\end{aligned}
        \end{equation}

        Estimates~\eqref{est:GGammaL2} and \eqref{est:L2norm2} yield
        \begin{equation*}
        	\begin{aligned}
        		&\left\|\frac{G\Gamma^{I_2}w}{\langle s-r\rangle^{\frac{3}{4}(\frac{1}{2}+\delta)}\langle r\rangle^{\frac{1}{4}}}\right\|
        		\lesssim(C_1\varepsilon)^{\frac{1}{4}}\left\|\frac{G\Gamma^{I_2}w}{\langle s-r\rangle^{\frac{1}{2}+\delta}}\right\|^{\frac{3}{4}}.
        	\end{aligned}
        \end{equation*}
        
        Substituting the above estimate into~\eqref{est:III2}, we obtain
        \begin{equation*}
        	\begin{aligned}
        		\mathcal{III}_2\lesssim C_1^2\varepsilon^{\frac{5}{4}}\langle s\rangle^{-\frac{5}{4}+\delta}+(C_1\varepsilon)^{\frac{5}{4}}\langle s\rangle^{-\frac{3}{4}}\sum_{\substack{|I_2|\leq|I|}}\left\|\frac{G\Gamma^{I_2}w}{\langle s-r\rangle^{\frac{1}{2}+\delta}}\right\|^{\frac{3}{4}}.
        	\end{aligned}
        \end{equation*}
        
	    By the H\"{o}lder inequality and \eqref{est:L2norm1}, we obtain
	    \begin{equation*}
	    		\int_{0}^{t}\|\partial\Gamma^IQ_0(w,w)\|\d s
	    		\lesssim C_1^2\varepsilon^{\frac{5}{4}}.
	    \end{equation*}
	\end{proof}

	\subsection{Proof of Proposition~\ref{pro:main}}\label{Se:End}
	In this subsection, we complete the proof of Proposition~\ref{pro:main} by refining all the estimates for $w$ in~\eqref{est:Bootw}.
	\begin{proof}[Proof of Proposition~\ref{pro:main}]
		For any initial data $(w_{0},w_{1})$ satisfying the conditions~\eqref{est:initial1}-\eqref{est:initial2}, we consider the corresponding solution $w$ of~\eqref{equ:Wave}. From these conditions, we observe that
			\begin{align}\label{est:initialenergy}
				\mathcal{E}_{gst}^{\frac{1}{2}}(0,\Gamma^I w)&\lesssim K^{|I|+1}, \ \quad |I|\leq N,\\
				\mathcal{E}_{gst}^{\frac{1}{2}}(0,\partial\Gamma^I w)&\lesssim\varepsilon K^{|I|+1}, \quad |I|\leq N-1.\label{est:initialenergy1}
			\end{align}
	
		\textbf{Step 1.} Closing the estimates for $\mathcal{E}_{gst}(t,\Gamma^{I}{w})$ with $|I|\leq N$.
		
		We first perform the highest-order energy estimate. 
		Let $N\geq5$, $|I|=N$. Acting $\Gamma^I$ on both sides of \eqref{equ:Wave}, we get
        \begin{equation}\label{equ:gamma}
        	\begin{aligned}
        		-\Box\Gamma^I w&=P^{\gamma\alpha\beta}\partial_\gamma w\partial_\alpha\partial_\beta\Gamma^I w
        		+\sum_{\substack{|I_1|+|I_2|\leq|I|\\|I_2|<|I|}}a^I_{I_1,I_2}Q_P(\Gamma^{I_{1}}w,\Gamma^{I_{2}}w)\\
        		&+\sum_{\substack{|I_1|+|I_2|\leq|I|}}b_{I_1,I_2}^IQ_0(\Gamma^{I_{1}}w,\Gamma^{I_{2}}w),
        	\end{aligned}
        \end{equation}
        where $a^I_{I_1,I_2}, b^I_{I_1,I_2}$ are constants.
        
        Recall $e^q=e^{q(r,t)} $, where $q(r,t)=\int_{-\infty}^{r-t}\langle s\rangle^{-1-2\delta}\d s$. Multiplying both sides of equation \eqref{equ:gamma} by $e^q\partial_t \Gamma^I w$, and employing integration by parts, we have
		\begin{equation}
			\begin{aligned}
				&\frac{1}{2}\partial_t \int_{\R^{3}}e^q|\partial\Gamma^I w|^2\d x+\frac{1}{2}\sum_{a=1, 2,3}\int_{\R^{3}}e^q\frac{|G_a\Gamma^I w|^2}{\langle r-t\rangle^{1+2\delta}}\d x\\
				=&\int_{\R^{3}}H^{\alpha\beta}\partial_\alpha\partial_\beta\Gamma^I w\partial_t\Gamma^I w e^q\d x\\
				+&\sum_{\substack{|I_1|+|I_2|\leq|I|\\|I_2|<|I|}}\int_{\R^{3}}a^I_{I_1,I_2}Q_P(\Gamma^{I_{1}}w,\Gamma^{I_{2}}w)\partial_t\Gamma^I w e^q\d x\\
				+&\sum_{\substack{|I_1|+|I_2|\leq|I|}}\int_{\R^{3}}b^I_{I_1,I_2}Q_0(\Gamma^{I_{1}}w,\Gamma^{I_{2}}w)\partial_t\Gamma^I w e^q\d x,\label{equ:energy}
			\end{aligned}
		\end{equation}
		where $H^{\alpha\beta}=P^{\gamma\alpha\beta}\partial_\gamma w.$ We now consider the first term on the right-hand side of the above equality. By applying the product rule, it can be rewritten as
		\begin{equation}
			\begin{aligned}
				&H^{\alpha\beta}\partial_\alpha\partial_\beta\Gamma^I w\partial_t\Gamma^I w e^q\\
				&=\partial_\alpha(H^{\alpha\beta}\partial_\beta\Gamma^I w\partial_t\Gamma^I w e^q)-\frac{1}{2}\partial_t(H^{\alpha\beta}\partial_\beta\Gamma^I w\partial_\alpha\Gamma^I w e^q)\\
				&-\partial_\alpha H^{\alpha\beta}\partial_\beta\Gamma^Iw \partial_t\Gamma^I w e^q-H^{\alpha\beta}\partial_\beta\Gamma^I w\partial_t \Gamma^I w\partial_\alpha e^q\\
				&+\frac{1}{2}\partial_tH^{\alpha\beta}\partial_\beta\Gamma^I w\partial_\alpha\Gamma^I w e^q+\frac{1}{2}H^{\alpha\beta}\partial_\beta\Gamma^I w \partial_\alpha\Gamma^I w\partial_t e^q. \label{equ:H}
			\end{aligned}
		\end{equation}
	By directly computing the fourth term in \eqref{equ:H}, we obtain
		\begin{equation}
			\begin{aligned}
				&-H^{\alpha\beta}\partial_\beta\Gamma^I w\partial_t\Gamma^I w\partial_\alpha e^q\\
				=&-H^{0\beta}\partial_\beta\Gamma^I w\partial_t\Gamma^I w\partial_t e^q-H^{a\beta}\partial_\beta\Gamma^I w\partial_t\Gamma^I w\partial_a e^q\\
				&-H^{a\beta}\partial_\beta\Gamma^I w\partial_a\Gamma^I w\partial_t e^q+H^{a\beta}\partial_\beta\Gamma^I w\partial_a\Gamma^I w\partial_t e^q\\
				=&-H^{\alpha\beta}\partial_\beta\Gamma^I w\partial_\alpha\Gamma^I w\partial_t e^q-H^{a\beta}\partial_\beta\Gamma^I wG_a\Gamma^I w \frac{e^q}{\langle r-t\rangle^{1+2\delta}}.\label{equ:H2}
			\end{aligned}
		\end{equation}
	
		By~\eqref{equ:energy},~\eqref{equ:H}, and~\eqref{equ:H2}, we obtain
		\begin{equation}\label{equ:integral}
			\begin{aligned}
				&\frac{1}{2}\partial_t \int_{\R^{3}}e^q|\partial\Gamma^I w|^2\d x-\partial_t\int_{\R^{3}}H^{0\beta}\partial_\beta\Gamma^I w\partial_t\Gamma^I we^q\d x\\
				&+\frac{1}{2}\partial_t\int_{\R^{3}}H^{\alpha\beta}\partial_\beta\Gamma^I w\partial_\alpha\Gamma^I we^q\d x
				+\frac{1}{2}\sum_{a=1, 2,3}\int_{\R^{3}}e^q\frac{|G_a\Gamma^I w|^2}{\langle r-t\rangle^{1+2\delta}}\d x\\
				=&\sum_{\substack{|I_1|+|I_2|\leq|I|\\|I_2|<|I|}}\int_{\R^{3}}a^I_{I_1,I_2}Q_P(\Gamma^{I_{1}}w,\Gamma^{I_{2}}w)\partial_t\Gamma^I w e^q\d x\\
				+&\sum_{\substack{|I_1|+|I_2|\leq|I|}}\int_{\R^{3}}b^I_{I_1,I_2}Q_0(\Gamma^{I_{1}}w,\Gamma^{I_{2}}w)\partial_t\Gamma^I w e^q\d x\\
				-&\int_{\R^{3}}\partial_\alpha H^{\alpha\beta}\partial_\beta\Gamma^Iw \partial_t\Gamma^I w e^q\d x-\int_{\R^{3}}H^{a\beta}\partial_\beta\Gamma^I wG_a\Gamma^I w\frac{e^q}{\langle r-t\rangle^{1+2\delta}}\d x\\
				+&\frac{1}{2}\int_{\R^{3}}\partial_tH^{\alpha\beta}\partial_\beta\Gamma^I w\partial_\alpha\Gamma^I w e^q\d x-\frac{1}{2}\int_{\R^{3}}H^{\alpha\beta}\partial_\beta\Gamma^I w \partial_\alpha\Gamma^I w\partial_t e^q\d x\\
				=&\mathcal{A}_1+\mathcal{A}_2+\mathcal{A}_3+\mathcal{A}_4+\mathcal{A}_5+\mathcal{A}_6,
			\end{aligned}
		\end{equation}
		and define 
		\begin{equation*}
			\begin{aligned}
				E_M(t,w)&=\int_{\R^{3}}e^q\left(|\partial w|^2-2H^{0\beta}\partial_\beta w\partial_t w
				+H^{\alpha\beta}\partial_\beta w\partial_\alpha w\right)\d x\\
				&+\sum_{a=1}^{3}\int_{0}^{t}\int_{\R^{3}}e^q\frac{|G_a w|^2}{\langle r-s\rangle^{1+2\delta}}\d x\d s.
			\end{aligned}
		\end{equation*}

		Next, we will estimate $\int_{0}^{t}|\mathcal{A}_1|\d s$ to $\int_{0}^{t}|\mathcal{A}_6|\d s$ successively.
		
		\textbf{Bound for $\int_{0}^{t}|\mathcal{A}_1|\d s$.} By Lemma~\ref{lem:P&Q}, we deduce
		\begin{equation}
			\begin{aligned}
				&\int_{0}^{t}|\mathcal{A}_1|\d s\lesssim\sum_{\substack{|I_1|+|I_2|\leq|I|\\|I_2|<|I|}}\int_{0}^{t}\int_{\R^{3}}|Q_P(\Gamma^{I_{1}}w,\Gamma^{I_{2}}w)\partial_t\Gamma^I w e^q|\d x\d s\\
				&\lesssim\sum_{\substack{|I_1|+|I_2|\leq|I|\\|I_2|<|I|}}\int_{0}^{t}\int_{\R^{3}}|G\Gamma^{I_1}w||\partial\partial\Gamma^{I_2}w||\partial\Gamma^Iw|\d x\d s\\
				&+\sum_{\substack{|I_1|+|I_2|\leq|I|\\|I_2|<|I|}}\int_{0}^{t}\int_{\R^{3}}|\partial\Gamma^{I_1}w||G\partial\Gamma^{I_2} w||\partial\Gamma^Iw|\d x\d s
				= \mathcal{A}_{11}+\mathcal{A}_{12}.
			\end{aligned}
		\end{equation}
	
	By \eqref{est:L2norm1}, \eqref{est:L2norm2}, \eqref{est:pointwise2}, \eqref{est:Ggamma}, and the H\"{o}lder inequality, we get
	    \begin{equation}
	    	\begin{aligned}
	    		\mathcal{A}_{11}&\lesssim\sum_{\substack{|I_1|+|I_2|\leq|I|\\|I_2|\leq N-3}}\int_{0}^{t}\left\|\frac{G\Gamma^{I_1}w}{\langle r-s\rangle^{\frac{1}{2}+\delta}}\right\|\|\langle r-s\rangle^{\frac{1}{2}+\delta}\partial\partial\Gamma^{I_2} w\|_{L^\infty}\|\partial\Gamma^Iw\|\d s\\
	    		&+\sum_{\substack{|I_1|+|I_2|\leq|I|\\|I_1|\leq N-3,|I_2|\leq |I|-1}}\int_{0}^{t}\|G\Gamma^{I_1}w\|_{L^\infty}\|\partial\partial\Gamma^{I_2}w\|\|\partial\Gamma^Iw\|\d s\\
	    		&\lesssim\sum_{|I_1|\leq|I|}\int_{0}^{t}C_1^2\varepsilon\langle s\rangle^{-1+\delta}\left\|\frac{G\Gamma^{I_1}w}{\langle r-s\rangle^{\frac{1}{2}+\delta}}\right\|\d s
	    		+C_1^3\varepsilon^{\frac{5}{4}}
	    		\lesssim C_1^3\varepsilon.
	    	\end{aligned}
	    \end{equation}
	    
    Using estimates~\eqref{est:L2norm1}, \eqref{est:L2norm2}, \eqref{est:Gpartialgamma}, and \eqref{est:partialrefined} (with $\lambda=\frac{1}{2}$), we have
        \begin{equation}
        	\begin{aligned}
        		\mathcal{A}_{12}&\lesssim\sum_{\substack{|I_1|+|I_2|\leq|I|\\|I_2|\leq N-4}}\int_{0}^{t}\|\partial\Gamma^{I_1}w\|\|G\partial\Gamma^{I_2}w\|_{L^\infty}\|\partial\Gamma^Iw\|\d s\\
        		&+\sum_{\substack{|I_1|+|I_2|\leq|I|\\|I_1|\leq N-3,|I_2|\leq |I|-1}}\int_{0}^{t}\|\langle r-s\rangle^{\frac{1}{2}+\delta}\partial\Gamma^{I_1}w\|_{L^\infty}\left\|\frac{G\partial\Gamma^{I_2}w}{\langle r-s\rangle^{\frac{1}{2}+\delta}}\right\|\|\partial\Gamma^Iw\|\d s\\
        		&\lesssim C_1^3\varepsilon+\sum_{|I_2|\leq|I|-1}\int_{0}^{t}C_1^2\varepsilon^
        		{\frac{1}{4}}\langle s\rangle^{-\frac{3}{4}+2\delta}\left\|\frac{G\partial\Gamma^{I_2}w}{\langle r-s\rangle^{\frac{1}{2}+\delta}}\right\|\d s
        		\lesssim C_1^3\varepsilon.
        	\end{aligned}
        \end{equation}
        
	    \textbf{Bound for $\int_{0}^{t}|\mathcal{A}_2|\d s$.} By Lemma~\ref{lem:P&Q} and estimates~\eqref{est:L2norm1}, \eqref{est:L2norm2}, \eqref{est:Ggamma}, and~\eqref{est:partialrefined} (with $\lambda=\frac{1}{2}$), we derive
	    \begin{equation}
	    	\begin{aligned}
	    		&\int_{0}^{t}|\mathcal{A}_2|\d s\lesssim\sum_{\substack{|I_1|+|I_2|\leq|I|}}\int_{0}^{t}\int_{\R^{3}}|G\Gamma^{I_1}w\partial\Gamma^{I_2}w\partial_t\Gamma^I w|\d x\d s\\
	    		&\lesssim\sum_{\substack{|I_1|+|I_2|\leq|I|\\|I_1|\leq N-3}}\int_{0}^{t}\|G\Gamma^{I_1}w\|_{L^\infty}\|\partial\Gamma^{I_2}w\|\|\partial\Gamma^Iw\|\d s\\
	    		&+\sum_{\substack{|I_1|+|I_2|\leq|I|\\|I_2|\leq N-3}}\int_{0}^{t}\left\|\frac{G\Gamma^{I_1}w}{\langle r-s\rangle^{\frac{1}{2}+\delta}}\right\|\|\langle r-s\rangle^{\frac{1}{2}+\delta}\partial\Gamma^{I_2}w\|_{L^\infty}\|\partial\Gamma^Iw\|\d s
	    		\lesssim C_1^3\varepsilon^{\frac{1}{4}}.
	    	\end{aligned}
	    \end{equation}
	    
	    \textbf{Bound for $\int_{0}^{t}|\mathcal{A}_3|\d s$.} By Lemma~\ref{lem:P1}, \eqref{est:L2norm1}, \eqref{est:pointwise2}, and \eqref{est:Gpartialgamma}, we have
		\begin{equation}
			\begin{aligned}
				&\int_{0}^{t}|\mathcal{A}_3|\d s
				\lesssim\int_{0}^{t}\int_{\R^{3}}|G\partial w||\partial\Gamma^Iw|^2\d x\d s+\int_{0}^{t}\int_{\R^{3}}|\partial\partial w||G\Gamma^Iw||\partial\Gamma^Iw|\d x\d s\\
				&\lesssim\int_{0}^{t}\|G\partial w\|_{L^\infty}\|\partial\Gamma^Iw\|^2\d s
				+\int_{0}^{t}\|\langle r-s\rangle^{\frac{1}{2}+\delta}\partial\partial w\|_{L^\infty}\left\|\frac{G\Gamma^Iw}{\langle r-s\rangle^{\frac{1}{2}+\delta}}\right\|\|\partial\Gamma^Iw\|\d s\\
				&\lesssim C_1^3\varepsilon+\int_{0}^{t}C_1^2\varepsilon\langle s\rangle^{-1+\delta}\left\|\frac{G\Gamma^Iw}{\langle r-s\rangle^{\frac{1}{2}+\delta}}\right\|\d s
				\lesssim C_1^3\varepsilon.
			\end{aligned}
		\end{equation}
	
	    \textbf{Bound for $\int_{0}^{t}|\mathcal{A}_4|\d s$.} By \eqref{est:L2norm1},~\eqref{est:partial}, and the H\"{o}lder inequality, we have
	    \begin{equation}
	    	\begin{aligned}
	    		\int_{0}^{t}|\mathcal{A}_4|\d s
	    		&\lesssim\int_{0}^{t}\int_{\R^{3}}|\partial w\partial\Gamma^IwG\Gamma^I w|\langle r-s\rangle^{-1-2\delta}\d x\d s\\
	    		&\lesssim\int_{0}^{t}C_1^2\varepsilon^{\frac{1}{2}}\langle s\rangle^{-1+\delta}\left\|\frac{G\Gamma^Iw}{\langle r-s\rangle^{\frac{1}{2}+\delta}}\right\|\d s
	    		\lesssim C_1^3\varepsilon^{\frac{1}{2}}.
	    	\end{aligned}
	    \end{equation}

		\textbf{Bound for $\int_{0}^{t}|\mathcal{A}_5|\d s$.} Similarly to $\mathcal{A}_3$, we deduce
		\begin{equation}
			\begin{aligned}
				\int_{0}^{t}|\mathcal{A}_5|\d s
				\lesssim&\int_{0}^{t}\int_{\R^{3}}|G\partial w||\partial\Gamma^Iw|^2\d x\d s+\int_{0}^{t}\int_{\R^{3}}|\partial\partial w||G\Gamma^Iw||\partial\Gamma^Iw|\d x\d s
				\lesssim C_1^3\varepsilon.
			\end{aligned}
		\end{equation}
	
	    \textbf{Bound for $\int_{0}^{t}|\mathcal{A}_6|\d s$.} By Lemma~\ref{lem:P1}, \eqref{est:L2norm1},~\eqref{est:Ggamma},~\eqref{est:partial}, and the H\"{o}lder inequality, we infer that
		\begin{equation}
			\begin{aligned}
				&\int_{0}^{t}|\mathcal{A}_6|\d s
				\lesssim\int_{0}^{t}\int_{\R^{3}}\left|P^{\gamma\alpha\beta}\partial_\gamma w\partial_\beta\Gamma^I w\partial_\alpha\Gamma^I w\frac{e^q}{\langle r-s\rangle^{1+2\delta}}\right|\d x\d s\\
				&\lesssim\int_{0}^{t}\int_{\R^{3}}|Gw||\partial\Gamma^Iw|^2\d x\d s+\int_{0}^{t}\int_{\R^{3}}|\partial w||G\Gamma^Iw||\partial\Gamma^Iw|\langle r-s\rangle^{-1-2\delta}\d x\d s\\
				&\lesssim C_1^3\varepsilon^{\frac{1}{4}}+\int_{0}^{t}C_1^2\varepsilon^{\frac{1}{2}}\langle s\rangle^{-1+\delta}\left\|\frac{G\Gamma^Iw}{\langle r-s\rangle^{\frac{1}{2}+\delta}}\right\|\d s
				\lesssim C_1^3\varepsilon^{\frac{1}{4}}.
			\end{aligned}
		\end{equation}

	    Thus, we have
	    \begin{equation*}
	    	E_M(t,\Gamma^Iw)\lesssim E_M(0,\Gamma^Iw)+C_1^3\varepsilon^{\frac{1}{4}}.
	    \end{equation*}
	    
	    Note that $e^q\sim1$. By \eqref{est:partial}, we derive
	    \begin{equation*}
	    	\begin{aligned}
	    		&\int_{\R^{3}}|H^{\alpha\beta}\partial_\beta\Gamma^I w\partial_\alpha\Gamma^I w-2H^{0\beta}\partial_\beta\Gamma^I w\partial_t\Gamma^I w|\d x\\
	    		&\lesssim\|H^{\alpha\beta}\|_{L^{\infty}}\int_{\R^{3}}|\partial_\alpha\Gamma^I w\partial_\beta\Gamma^I w|\d x\\
	    		&\lesssim\|\partial w\|_{L^\infty}\int_{\R^{3}}|\partial\Gamma^I w|^2\d x\lesssim C_1\varepsilon\int_{\R^{3}}|\partial\Gamma^I w|^2\d x,
	    	\end{aligned}
	    \end{equation*}
	     which implies
	    \begin{equation*}
	    	E_M(t,\Gamma^Iw)\sim \mathcal{E}_{gst}(t,\Gamma^Iw).
	    \end{equation*}
    
		Therefore, the above inequalities and \eqref{est:initialenergy} yield
		\begin{equation}\label{est:highest1}
			\mathcal{E}_{gst}(t,\Gamma^Iw)\lesssim \mathcal{E}_{gst}(0,\Gamma^Iw)+C_1^3\varepsilon^{\frac{1}{4}}\lesssim K^{2N+2}+C_1^3\varepsilon^{\frac{1}{4}}, \quad\mbox{for}\quad |I|=N.
		\end{equation}
		
		Now we perform the lower-order ghost weight energy estimate. Let $N\geq5$ and $|I|\leq N-1$. Using~\eqref{est:ghost}, \eqref{est:initialenergy}, and Lemma~\ref{est:L1L2nonlinear}, we deduce that
		\begin{equation}\label{est:lower1}
			\begin{aligned}
				\mathcal{E}_{gst}^{\frac{1}{2}}(t,\Gamma^Iw)&\lesssim \mathcal{E}_{gst}^{\frac{1}{2}}(0,\Gamma^Iw)+\sum_{|J|\leq |I|}\int_{0}^{t}\left(\|\Gamma^JQ_P(w,w)\|+\|\Gamma^JQ_0(w,w)\|\right)\d s\\
				&\lesssim K^{N}+C_1^2\varepsilon^{\frac{1}{4}}.
			\end{aligned}
		\end{equation}
		Based on (\ref{est:highest1}) and (\ref{est:lower1}), we deduce that these estimates strictly improve the bootstrap bound for $\mathcal{E}_{gst}(t,\Gamma^I w)$ in (\ref{est:Bootw}), provided that $C_1$ is sufficiently large and $\varepsilon$ is sufficiently small (depending on $C_1$). 
		
		\textbf{Step 2.} Closing the estimate for $\mathcal{E}_{gst}(t,\partial\Gamma^{I}{w})$ with $|I|\leq N-1$.
		
		We first perform the highest-order energy estimate. Let $N\geq5$, and $|I|=N-1$. Acting $\partial\Gamma^I$ on both sides of the equation~\eqref{equ:Wave}, we get
		\begin{equation}\label{equ:partialgamma}
			\begin{aligned}
				-\Box\partial\Gamma^I w&=P^{\gamma\alpha\beta}\partial_\gamma w\partial_\alpha\partial_\beta\partial\Gamma^I w+ \sum_{\substack{|I_1|+|I_2|\leq|I|}}a^I_{I_1,I_2}Q_P(\partial\Gamma^{I_{1}}w,\Gamma^{I_{2}}w)\\
				&+\sum_{\substack{|I_1|+|I_2|\leq|I|\\|I_2|<|I|}}a^I_{I_1,I_2}Q_P(\Gamma^{I_{1}}w,\partial\Gamma^{I_{2}}w)+\sum_{\substack{|J|\leq|I|}}b_J^I\partial\Gamma^JQ_0(w,w),
			\end{aligned}
		\end{equation}
	    where $a^I_{I_1,I_2}, b^I_J$ are constants.

	    Multiplying (\ref{equ:partialgamma}) by $e^q\partial_t\partial\Gamma^I w$, and employing integration by parts, we have
	    \begin{equation*}
	    	\begin{aligned}
	    		&\frac{1}{2}\partial_t \int_{\R^{3}}e^q|\partial\partial\Gamma^I w|^2\d x+\frac{1}{2}\sum_{a=1, 2,3}\int_{\R^{3}}e^q\frac{|G_a\partial\Gamma^I w|^2}{\langle r-t\rangle^{1+2\delta}}\d x\\
	    		=&\int_{\R^{3}}H^{\alpha\beta}\partial_\alpha\partial_\beta\partial\Gamma^I w\partial_t\partial\Gamma^I w e^q\d x\\
	    		+&\sum_{\substack{|I_1|+|I_2|\leq|I|}}\int_{\R^{3}}a^I_{I_1,I_2}Q_P(\partial\Gamma^{I_{1}}w,\Gamma^{I_{2}}w)\partial_t\partial\Gamma^I w e^q\d x\\
	    		+&\sum_{\substack{|I_1|+|I_2|\leq|I|\\|I_2|<|I|}}\int_{\R^{3}}a^I_{I_1,I_2}Q_P(\Gamma^{I_{1}}w,\partial\Gamma^{I_{2}}w)\partial_t\partial\Gamma^I w e^q\d x\\
	    		+&\sum_{\substack{|J|\leq|I|}}\int_{\R^{3}}b^I_{J}\partial\Gamma^JQ_0(w,w)\partial_t\partial\Gamma^I w e^q\d x,
	    	\end{aligned}
	    \end{equation*}
    where $H^{\alpha\beta}=P^{\gamma\alpha\beta}\partial_\gamma w$.
        
        The first term on the right-hand side can be rewritten as
        \begin{equation*}
        	\begin{aligned}
        		&H^{\alpha\beta}\partial_\alpha\partial_\beta\partial\Gamma^I w\partial_t\partial\Gamma^I w e^q\\
        		=&\partial_\alpha(H^{\alpha\beta}\partial_\beta\partial\Gamma^I w\partial_t\partial\Gamma^I w e^q)-\frac{1}{2}\partial_t(H^{\alpha\beta}\partial_\beta\partial\Gamma^I w\partial_\alpha\partial\Gamma^I w e^q)\\
        		&-\partial_\alpha H^{\alpha\beta}\partial_\beta\partial\Gamma^Iw \partial_t\partial\Gamma^I w e^q-H^{a\beta}\partial_\beta\partial\Gamma^I wG_a\partial\Gamma^I w \frac{e^q}{\langle r-t\rangle^{1+2\delta}}\\
        		&+\frac{1}{2}\partial_tH^{\alpha\beta}\partial_\beta\partial\Gamma^I w\partial_\alpha\partial\Gamma^I w e^q-\frac{1}{2}H^{\alpha\beta}\partial_\beta\partial\Gamma^I w \partial_\alpha\partial\Gamma^I w\partial_t e^q.
        	\end{aligned}
        \end{equation*}

        Similarly to \eqref{equ:integral}, we obtain
        \begin{equation*}
        	\begin{aligned}
        		&\frac{1}{2}\partial_t \int_{\R^{3}}e^q|\partial\partial\Gamma^I w|^2\d x-\partial_t\int_{\R^{3}}H^{0\beta}\partial_\beta\partial\Gamma^I w\partial_t\partial\Gamma^I we^q\d x\\
        		&+\frac{1}{2}\partial_t\int_{\R^{3}}H^{\alpha\beta}\partial_\beta\partial\Gamma^I w\partial_\alpha\partial\Gamma^I we^q\d x
        		+\frac{1}{2}\sum_{a=1, 2,3}\int_{\R^{3}}e^q\frac{|G_a\partial\Gamma^I w|^2}{\langle r-t\rangle^{1+2\delta}}\d x\\
        		&=\sum_{\substack{|I_1|+|I_2|\leq|I|}}\int_{\R^{3}}a^I_{I_1,I_2}Q_P(\partial\Gamma^{I_{1}}w,\Gamma^{I_{2}}w)\partial_t\partial\Gamma^I w e^q\d x\\
        		&+\sum_{\substack{|I_1|+|I_2|\leq|I|\\|I_2|<|I|}}\int_{\R^{3}}a^I_{I_1,I_2}Q_P(\Gamma^{I_{1}}w,\partial\Gamma^{I_{2}}w)\partial_t\partial\Gamma^I w e^q\d x\\
        		&+\sum_{\substack{|J|\leq|I|}}\int_{\R^{3}}b^I_{J}\partial\Gamma^JQ_0(w,w)\partial_t\partial\Gamma^I w e^q\d x\\
        		&-\int_{\R^{3}}\partial_\alpha H^{\alpha\beta}\partial_\beta\partial\Gamma^Iw \partial_t\partial\Gamma^I w e^q\d x-\int_{\R^{3}}H^{a\beta}\partial_\beta\partial\Gamma^I wG_a\partial\Gamma^I w\frac{e^q}{\langle r-t\rangle^{1+2\delta}}\d x\\
        		&+\frac{1}{2}\int_{\R^{3}}\partial_tH^{\alpha\beta}\partial_\beta\partial\Gamma^I w\partial_\alpha\partial\Gamma^I w e^q\d x-\frac{1}{2}\int_{\R^{3}}H^{\alpha\beta}\partial_\beta\partial\Gamma^I w \partial_\alpha\partial\Gamma^I w\partial_t e^q\d x\\
        		&=\mathcal{B}_1+\mathcal{B}_2+\mathcal{B}_3+\mathcal{B}_4+\mathcal{B}_5+\mathcal{B}_6+\mathcal{B}_7.
        	\end{aligned}
        \end{equation*}
         
         Next, we successively estimate $\int_{0}^{t}|\mathcal{B}_1|\d s$ through $\int_{0}^{t}|\mathcal{B}_7|\d s$.
         
         \textbf{Bound for $\int_{0}^{t}|\mathcal{B}_1|\d s$.} Based on Lemma~\ref{lem:P&Q}, \eqref{est:L2norm2}, \eqref{est:Gpartialgamma}, and \eqref{est:pointwise2}, we obtain
        \begin{equation}
        	\begin{aligned}
        		&\int_{0}^{t}|\mathcal{B}_1|\d s\lesssim\sum_{\substack{|I_1|+|I_2|\leq|I|}}\int_{0}^{t}\int_{\R^{3}}|\partial\partial\Gamma^{I_1}w||G\partial\Gamma^{I_2}w||\partial\partial\Gamma^Iw|\d x\d s\\
        		&\lesssim\sum_{\substack{|I_1|+|I_2|\leq|I|\\|I_2|\leq N-4}}\int_{0}^{t}\|\partial\partial\Gamma^{I_1}w\|\|G\partial\Gamma^{I_2}w\|_{L^\infty}\|\partial\partial\Gamma^Iw\|\d s\\
        		&+\sum_{\substack{|I_1|+|I_2|\leq|I|\\|I_1|\leq N-3}}\int_{0}^{t}\|\langle r-s\rangle^{\frac{1}{2}+\delta}\partial\partial\Gamma^{I_1}w\|_{L^\infty}\left\|\frac{G\partial\Gamma^{I_2}w}{\langle r-s\rangle^{\frac{1}{2}+\delta}}\right\|\|\partial\partial\Gamma^Iw\|\d s\\
        		&\lesssim (C_1\varepsilon)^3+\sum_{\substack{|I_2|\leq|I|}}\int_{0}^{t}C_1^2\varepsilon^2\langle s\rangle^{-1+\delta}\left\|\frac{G\partial\Gamma^{I_2}w}{\langle r-s\rangle^{\frac{1}{2}+\delta}}\right\|\d s
        		\lesssim C_1^3\varepsilon^3.
        	\end{aligned}
        \end{equation}
    
        \textbf{Bound for $\int_{0}^{t}|\mathcal{B}_2|\d s$.} Using Lemma~\ref{lem:P&Q}, we derive
        \begin{equation}
        	\begin{aligned}
        		&\int_{0}^{t}|\mathcal{B}_2|\d s\lesssim\sum_{\substack{|I_1|+|I_2|\leq|I|\\|I_2|<|I|}}\int_{0}^{t}\int_{\R^{3}}|G\Gamma^{I_1}w||\partial\partial\partial\Gamma^{I_2}w||\partial\partial\Gamma^Iw|\d x\d s\\
        		&+\sum_{\substack{|I_1|+|I_2|\leq|I|\\|I_2|<|I|}}\int_{0}^{t}\int_{\R^{3}}|\partial\Gamma^{I_1}w||G\partial\partial\Gamma^{I_2}w||\partial\partial\Gamma^Iw|\d x\d s=\mathcal{B}_{21}+\mathcal{B}_{22}.\\
        	\end{aligned}
        \end{equation}

        Applying the H\"{o}lder inequality, we get
        \begin{equation}
        	\begin{aligned}
        		&\mathcal{B}_{21}\lesssim\sum_{\substack{|I_1|+|I_2|\leq|I|\\|I_1|\leq N-3,|I_2|<|I|}}\int_{0}^{t}\|G\Gamma^{I_1}w\|_{L^\infty}\|\partial\partial\partial\Gamma^{I_2}w\|\|\partial\partial\Gamma^Iw\|\d s\\
        		&+\sum_{\substack{|I_1|+|I_2|\leq|I|\\|I_2|\leq N-4}}\int_{0}^{t}\left\|\frac{G\Gamma^{I_1}w}{\langle r-s\rangle^{\frac{3}{4}(\frac{1}{2}+\delta)}\langle r\rangle^{\frac{1}{4}}}\right\|\|\langle r-s\rangle^{\frac{3}{4}(\frac{1}{2}+\delta)}\langle r\rangle^{\frac{1}{4}}\partial\partial\partial\Gamma^{I_2}w\|_{L^\infty}\|\partial\partial\Gamma^Iw\|\d s,
        	\end{aligned}
        \end{equation}
        where
        \begin{equation}\label{est:L2GGamma}
        	\begin{aligned}
        		&\left\|\frac{G\Gamma^{I_1}w}{\langle r-s\rangle^{\frac{3}{4}(\frac{1}{2}+\delta)}\langle r\rangle^{\frac{1}{4}}}\right\|\lesssim\left\|\frac{G\Gamma^{I_1}w}{\langle r\rangle}\right\|^{\frac{1}{4}}\left\|\frac{G\Gamma^{I_1}w}{\langle r-s\rangle^{\frac{1}{2}+\delta}}\right\|^{\frac{3}{4}}\\
        		&\lesssim\|\partial_rG\Gamma^{I_1}w\|^{\frac{1}{4}}\left\|\frac{G\Gamma^{I_1}w}{\langle r-s\rangle^{\frac{1}{2}+\delta}}\right\|^{\frac{3}{4}}
        		\lesssim\|\partial\partial\Gamma^{I_1}w\|^{\frac{1}{4}}\left\|\frac{G\Gamma^{I_1}w}{\langle r-s\rangle^{\frac{1}{2}+\delta}}\right\|^{\frac{3}{4}}.
        	\end{aligned}
        \end{equation}
        Thus, by \eqref{est:L2norm1}, \eqref{est:L2norm2}, \eqref{est:pointwise2}, \eqref{est:Ggamma}, and the above inequalities, we have
        \begin{equation}
        	\begin{aligned}
        		\mathcal{B}_{21}&\lesssim C_1^3\varepsilon^{\frac{9}{4}}+(C_1\varepsilon)^{\frac{9}{4}}\sum_{\substack{|I_1|\leq|I|}}\int_{0}^{t}\langle s\rangle^{-\frac{3}{4}}\left\|\frac{G\Gamma^{I_1}w}{\langle r-s\rangle^{\frac{1}{2}+\delta}}\right\|^{\frac{3}{4}}\d s\\
        		&\lesssim C_1^3\varepsilon^{\frac{9}{4}}+ (C_1\varepsilon)^{\frac{9}{4}}\sum_{\substack{|I_1|\leq|I|}}\left(\int_{0}^{t}\left\|\frac{G\Gamma^{I_1}w}{\langle r-s\rangle^{\frac{1}{2}+\delta}}\right\|^{2}\d s\right)^{\frac{3}{8}}\left(\int_{0}^{t}\langle s\rangle^{-\frac{6}{5}}\d s\right)^{\frac{5}{8}}\lesssim C_1^3\varepsilon^{\frac{9}{4}}.
        	\end{aligned}
        \end{equation}
        
        Next, we bound $\mathcal{B}_{22}$. Based on Lemma~\ref{lem:extra}, we have 
        \begin{equation*}
        	|G\partial\partial\Gamma^{I_2}w|\lesssim\langle s+r\rangle^{-1}\sum_{|J|=1}|\Gamma^J\partial\partial\Gamma^{I_2}w|.
        \end{equation*}
        Thus, combining the above inequality with \eqref{est:partial}, \eqref{est:Gpartial2gamma}, we get
        \begin{equation}
        	\begin{aligned}
        		\mathcal{B}_{22}&\lesssim\sum_{\substack{|I_1|+|I_2|\leq|I|\\|I_1|\leq N-3,|I_2|\leq |I|-1\\|J|\leq 1}}\int_{0}^{t}\|\langle s+r\rangle^{-1}\partial\Gamma^{I_1}w\|_{L^\infty}\|\partial\partial\Gamma^J\Gamma^{I_2}w\|\|\partial\partial\Gamma^Iw\|\d s\\
        		&+\sum_{\substack{|I_1|+|I_2|\leq|I|\\|I_2|\leq N-4}}\int_{0}^{t}\left\|\frac{\partial\Gamma^{I_1}w}{\langle r\rangle^{\frac{1}{4}}}\right\|\|\langle r\rangle^{\frac{1}{4}} G\partial\partial\Gamma^{I_2}w\|_{L^\infty}\|\partial\partial\Gamma^Iw\|\d s\\
        		&\lesssim (C_1\varepsilon)^3+(C_1\varepsilon)^2\sum_{\substack{|I_1|\leq|I|}}\int_{0}^{t}\langle s\rangle^{-\frac{7}{4}}\left\|\frac{\partial\Gamma^{I_1}w}{\langle r\rangle}\right\|^{\frac{1}{4}}\|\partial\Gamma^{I_1}w\|^{\frac{3}{4}}\d s\lesssim C_1^3\varepsilon^{\frac{9}{4}},
        	\end{aligned}
        \end{equation}
        where we used the H\"{o}lder inequality and the Hardy inequality.
       
        \textbf{Bound for $\int_{0}^{t}|\mathcal{B}_3|\d s$.} Using Lemma~\ref{lem:P&Q}, we derive
        \begin{equation}
        	\begin{aligned}
        		&\int_{0}^{t}|\mathcal{B}_3|\d s\lesssim\sum_{\substack{|J_1|+|J_2|\leq|I|}}\int_{0}^{t}\int_{\R^{3}}|G\partial\Gamma^{J_1}w\partial\Gamma^{J_2}w\partial_t\partial\Gamma^Iw|\d x\d s\\
        		&+\sum_{\substack{|J_1|+|J_2|\leq|I|}}\int_{0}^{t}\int_{\R^{3}}|\partial\partial\Gamma^{J_1}wG\Gamma^{J_2}w\partial_t\partial\Gamma^Iw|\d x\d s
        		=\mathcal{B}_{31}+\mathcal{B}_{32}.
        	\end{aligned}
        \end{equation}
        
        By the H\"{o}lder inequality, the Hardy inequality, and estimates~\eqref{est:L2norm1}, \eqref{est:L2norm2}, \eqref{est:partialrefined},~\eqref{est:Gpartialgamma}, we obtain
        \begin{equation}
        	\begin{aligned}
        		&\mathcal{B}_{31}\lesssim\sum_{\substack{|J_1|+|J_2|\leq|I|\\|J_1|\leq N-4}}\int_{0}^{t}\|\langle r\rangle^{\frac{1}{4}}G\partial\Gamma^{J_1}w\|_{L^\infty}\left\|\frac{\partial\Gamma^{J_2}w}{\langle r\rangle^{\frac{1}{4}}}\right\|\|\partial_t\partial\Gamma^Iw\|\d s\\
        		&+\sum_{\substack{|J_1|+|J_2|\leq|I|\\|J_2|\leq N-3}}\int_{0}^{t}\left\|\frac{G\partial\Gamma^{J_1}w}{\langle r-s\rangle^{\frac{1}{2}+\delta}}\right\|\|\langle r-s\rangle^{\frac{1}{2}+\delta}\partial\Gamma^{J_2}w\|_{L^\infty}\|\partial_t\partial\Gamma^Iw\|\d s\\
        		&\lesssim(C_1\varepsilon)^2\sum_{\substack{|J_2|\leq|I|}}\int_{0}^{t}\langle s\rangle^{-\frac{5}{4}+\delta}\|\partial\partial\Gamma^{J_2}w\|^{\frac{1}{4}}\|\partial\Gamma^{J_2}w\|^{\frac{3}{4}}\d s\\
        		&+C_1^2\varepsilon^{\frac{5}{4}}\sum_{\substack{|J_1|\leq|I|}}\int_{0}^{t}\langle s\rangle^{-\frac{3}{4}+2\delta}\left\|\frac{G\partial\Gamma^{J_1}w}{\langle r-s\rangle^{\frac{1}{2}+\delta}}\right\|\d s
        		\lesssim C_1^3\varepsilon^{\frac{9}{4}}.
        	\end{aligned}
        \end{equation}
        
        Similarly to $\mathcal{B}_{21}$, we have
        \begin{equation}
        	\begin{aligned}
        		&\mathcal{B}_{32}\lesssim\sum_{\substack{|J_1|+|J_2|\leq|I|\\|J_2|\leq N-3}}\int_{0}^{t}\|\partial\partial\Gamma^{J_1}w\|\|G\Gamma^{J_2}w\|_{L^\infty}\|\partial\partial\Gamma^Iw\|\d s\\
        		&+\sum_{\substack{|J_1|+|J_2|\leq|I|\\|J_1|\leq N-3}}\int_{0}^{t}\|\langle r-s\rangle^{\frac{3}{4}(\frac{1}{2}+\delta)}\langle r\rangle^{\frac{1}{4}}\partial\partial\Gamma^{J_1}w\|_{L^\infty}\left\|\frac{G\Gamma^{J_2}w}{\langle r-s\rangle^{\frac{3}{4}(\frac{1}{2}+\delta)}\langle r\rangle^{\frac{1}{4}}}\right\|\|\partial\partial\Gamma^Iw\|\d s\\
        		&\lesssim C_1^3\varepsilon^{\frac{9}{4}}+(C_1\varepsilon)^{\frac{9}{4}}\sum_{\substack{|J_2|\leq|I|}}\int_{0}^{t}\langle s\rangle^{-\frac{3}{4}}\left\|\frac{G\Gamma^{J_2}w}{\langle r-s\rangle^{\frac{1}{2}+\delta}}\right\|^{\frac{3}{4}}\d s\lesssim C_1^3\varepsilon^{\frac{9}{4}}.
        	\end{aligned}
        \end{equation}

        \textbf{Bound for $\int_{0}^{t}|\mathcal{B}_4|\d s$.} By Lemma~\ref{lem:P1}, the H\"{o}lder inequality, and estimates~\eqref{est:L2norm2}, \eqref{est:pointwise2}, \eqref{est:Gpartialgamma}, we get
        \begin{equation}
        	\begin{aligned}
        		&\int_{0}^{t}|\mathcal{B}_4|\d s\lesssim\int_{0}^{t}\int_{\R^{3}}|\partial\partial w||G\partial\Gamma^Iw||\partial\partial\Gamma^Iw|\d x\d s
        		+\int_{0}^{t}\int_{\R^{3}}|G\partial w||\partial\partial\Gamma^Iw|^2\d x\d s\\
        		&\lesssim\int_{0}^{t}\|\langle r-s\rangle^{\frac{1}{2}+\delta}\partial\partial w\|_{L^\infty}\left\|\frac{G\partial\Gamma^Iw}{\langle r-s\rangle^{\frac{1}{2}+\delta}}\right\|\|\partial\partial\Gamma^Iw\|\d s
        		+\int_{0}^{t}\|G\partial w\|_{L^\infty}\|\partial\partial\Gamma^Iw\|^2\d s\\
        		&\lesssim (C_1\varepsilon)^{2}\int_{0}^{t}\langle s\rangle^{-1+\delta}\left\|\frac{G\partial\Gamma^Iw}{\langle r-s\rangle^{\frac{1}{2}+\delta}}\right\|\d s+(C_1\varepsilon)^2\int_{0}^{t}\|G\partial w\|_{L^\infty}\d s\lesssim (C_1\varepsilon)^3.
        	\end{aligned}
        \end{equation}

        \textbf{Bound for $\int_{0}^{t}|\mathcal{B}_5|\d s$.} Applying the H\"{o}lder inequality and estimates \eqref{est:L2norm2}, \eqref{est:partial}, we get
        \begin{equation}
        	\begin{aligned}
        		\int_{0}^{t}|\mathcal{B}_5|\d s&\lesssim\int_{0}^{t}\int_{\R^{3}}|\partial w||\partial\partial\Gamma^Iw||G\partial\Gamma^Iw|\langle r-s\rangle^{-1-2\delta}\d x\d s\\
        		&\lesssim\int_{0}^{t}\|\partial w\|_{L^\infty}\|\partial\partial\Gamma^Iw\|\left\|\frac{G\partial\Gamma^Iw}{\langle r-s\rangle^{\frac{1}{2}+\delta}}\right\|\d s\\
        		&\lesssim\int_{0}^{t}C_1^2\varepsilon^{\frac{3}{2}}\langle s\rangle^{-1+\delta}\left\|\frac{G\partial\Gamma^Iw}{\langle r-s\rangle^{\frac{1}{2}+\delta}}\right\|\d s\lesssim C_1^3\varepsilon^{\frac{5}{2}}.
        	\end{aligned}
        \end{equation}
        
        \textbf{Bound for $\int_{0}^{t}|\mathcal{B}_6|\d s$.} Lemma~\ref{lem:P1} gives
        \begin{equation*}
        	\begin{aligned}
        		\int_{0}^{t}|\mathcal{B}_6|\d s\lesssim\int_{0}^{t}\int_{\R^{3}}|\partial\partial w||G\partial\Gamma^Iw||\partial\partial\Gamma^Iw|\d x\d s
        		+\int_{0}^{t}\int_{\R^{3}}|G\partial w||\partial\partial\Gamma^Iw|^2\d x\d s.
        	\end{aligned}
        \end{equation*}
        Similarly to $\mathcal{B}_4$, we can obtain
        \begin{equation}
        	\int_{0}^{t}|\mathcal{B}_6|\d s\lesssim(C_1\varepsilon)^3.
        \end{equation}

        \textbf{Bound for $\int_{0}^{t}|\mathcal{B}_7|\d s$.} By Lemma~\ref{lem:P1}, the H\"{o}lder inequality, and estimates \eqref{est:L2norm2}, \eqref{est:partial}, \eqref{est:Ggamma}, we derive
        \begin{equation}
        	\begin{aligned}
        		\int_{0}^{t}|\mathcal{B}_7|\d s&\lesssim\int_{0}^{t}\int_{\R^{3}}|\partial w||\partial\partial\Gamma^Iw||G\partial\Gamma^Iw|\langle r-s\rangle^{-1-2\delta}\d x\d s\\
        		&+\int_{0}^{t}\int_{\R^{3}}|Gw||\partial\partial\Gamma^Iw|^2\langle r-s\rangle^{-1-2\delta}\d x\d s\\
        		&\lesssim\int_{0}^{t}\|\partial w\|_{L^\infty}\|\partial\partial\Gamma^Iw\|\left\|\frac{G\partial\Gamma^Iw}{\langle r-s\rangle^{\frac{1}{2}+\delta}}\right\|\d s+\int_{0}^{t}\|Gw\|_{L^\infty}\|\partial\partial\Gamma^Iw\|^2\d s\\
        		&\lesssim C_1^3\varepsilon^{\frac{9}{4}}.
        	\end{aligned}
        \end{equation}
        
        Thus, we have
        \begin{equation*}
        	E_M(t,\partial\Gamma^Iw)\lesssim E_M(0,\partial\Gamma^Iw)+C_1^3\varepsilon^{\frac{9}{4}}.
        \end{equation*}
        
        By estimate \eqref{est:partial}, we have 
        \begin{equation*}
        	\begin{aligned}
        		&\int_{\R^{3}}|H^{\alpha\beta}\partial_\beta\partial\Gamma^I w\partial_\alpha\partial\Gamma^I w-2H^{0\beta}\partial_\beta\partial\Gamma^I w\partial_t\partial\Gamma^I w|\d x\\
        		&\lesssim\|H^{\alpha\beta}\|_{L^{\infty}}\int_{\R^{3}}|\partial_\alpha\partial\Gamma^I w\partial_\beta\partial\Gamma^I w|\d x\\
        		&\lesssim\|\partial w\|_{L^{\infty}}\int_{\R^{3}}|\partial\partial\Gamma^I w|^2\d x
        		\lesssim C_1\varepsilon\int_{\R^{3}}|\partial\partial\Gamma^I w|^2\d x,
        	\end{aligned}
        \end{equation*}
        which implies
        \begin{equation*}
        	E_M(t,\partial\Gamma^I w)\sim\mathcal{E}_{gst}(t,\partial\Gamma^I w) .
        \end{equation*}

        Therefore, the above computations, together with \eqref{est:initialenergy1}, yield
        \begin{equation*}
        	\mathcal{E}_{gst}(t,\partial\Gamma^Iw)\lesssim \mathcal{E}_{gst}(0,\partial\Gamma^Iw)+C_1^3\varepsilon^{\frac{9}{4}}\lesssim\varepsilon^2 K^{2N}+C_1^3\varepsilon^{\frac{9}{4}}.
        \end{equation*}
        
        Now we perform the lower-order ghost weight energy estimate. Let $N\geq5$, $|I|\leq N-2$. Based on \eqref{est:ghost}, \eqref{est:initialenergy1}, and Lemma~\ref{est:L1L2nonlinear}, it follows that
        \begin{equation*}
        	\begin{aligned}
        		\mathcal{E}_{gst}^{\frac{1}{2}}(t,\partial\Gamma^Iw)&\lesssim \mathcal{E}_{gst}^{\frac{1}{2}}(0,\partial\Gamma^Iw)+\sum_{|J|\leq |I|}\int_{0}^{t}\left(\|\partial\Gamma^JQ_P(w,w)\|+\|\partial\Gamma^JQ_0(w,w)\|\right)\d s\\
        		&\lesssim\varepsilon K^{N-1}+C_1^2\varepsilon^{\frac{5}{4}}.
        	\end{aligned}
        \end{equation*}
        
		These results strictly improve the bootstrap bound for $\mathcal{E}_{gst}(t,\partial\Gamma^I w)$ in~\eqref{est:Bootw}, provided $C_1$ is sufficiently large and $\varepsilon$ is sufficiently small (depending on $C_1$).
		
		Therefore, we have improved all the bootstrap estimates of $w$ in \eqref{est:Bootw}. In other words, we have completed the proof of Proposition~\ref{pro:main}.
	\end{proof}

    \subsection{Scattering of the solution $w$}
    To complete the proof of Theorem~\ref{thm:main1}, we need to show that the solution $w$ scatters linearly.

    Proposition~\ref{pro:main} and Lemma~\ref{est:L1L2nonlinear} give
    \begin{align*}
    	\int_{0}^{+\infty}\|Q_P(w,w)\|_{\dot{H}^{N-1}(\mathbb{R}^3)}\d s\lesssim\sum_{|I|=N-1}\int_{0}^{+\infty}\|\nabla^IQ_P(w,w)\|\d s\lesssim C_1^2\varepsilon,\\
    	\int_{0}^{+\infty}\|Q_0(w,w)\|_{\dot{H}^{N-1}(\mathbb{R}^3)}\d s\lesssim\sum_{|I|=N-1}\int_{0}^{+\infty}\|\nabla^IQ_0(w,w)\|\d s\lesssim C_1^2\varepsilon^{\frac{1}{4}}.
    \end{align*}
    
    Based on Lemma~\ref{lem:scatterWave} and the above estimates, we conclude that $w$ scatters linearly. The proof of Theorem~\ref{thm:main1} is completed.

	\section{Uniform boundedness of conformal energy}\label{Se:Con-energy}
	
	We note that in Section~\ref{Se:exist}, we have completed the proof of the global existence of the solution to \eqref{equ:Wave} and the uniform boundedness of the ghost weight energy. To complete the proof of Theorem \ref{thm:main2}, it remains to establish that the lower-order conformal energy is also uniformly bounded.

	For any initial data $(w_{0},w_{1})$ satisfying the conditions in Theorem~\ref{thm:main2}, we consider the corresponding solution $w$ of~\eqref{equ:Wave}. From the condition~\eqref{est:initialconf}, we see that 

		\begin{align}
			\sum_{|I|\leq N-1}\left(\mathcal{E}_{con}^{\frac{1}{2}}(0,\Gamma^Iw)+\|\Gamma^Iw(0,x)\|+\|\partial_t\Gamma^I w(0,x)\|_{L^1\cap L^2}\right)\leq C_2,\label{est:conf}
		\end{align}
for some constant $C_2$.

    \subsection{Rough bounds for conformal energy}
    In this subsection, we first obtain rough bounds for conformal energy. Based on estimate~\eqref{est:Conformal}, we need to control the weighted $L^2$ estimates on the nonlinear terms.
    
    Let $|I|\leq N-1$. By Lemma~\ref{lem:P&Q}, we deduce that
    \begin{equation*}
    	\begin{aligned}
    		&\|\langle s+r\rangle\Gamma^IQ_P(w,w)\|\\
    		&\lesssim\sum_{\substack{|I_1|+|I_2|\leq|I|}}(\|\langle s+r\rangle G\Gamma^{I_1} w\partial\partial\Gamma^{I_2} w\|+\|\langle s+r\rangle\partial\Gamma^{I_1} wG\partial\Gamma^{I_2} w\|)
    		=\mathcal{K}_1+\mathcal{K}_2.
    	\end{aligned}
    \end{equation*}
    
    By the H\"{o}lder inequality, \eqref{est:L2norm2}, \eqref{est:Ggamma}, \eqref{est:partialdouble}, and Lemma~\ref{lem:extra}, we obtain
    \begin{equation*}
    	\begin{aligned}
    		\mathcal{K}_1&\lesssim\sum_{\substack{|I_1|+|I_2|\leq|I|\\|I_1|\leq N-3}}\|\langle s+r\rangle G\Gamma^{I_1}w\|_{L^\infty}\|\partial\partial\Gamma^{I_2}w\|\\
    		&+\sum_{\substack{|I_1|+|I_2|\leq|I|\\|I_2|\leq N-3}}\left\|\frac{G\Gamma^{I_1}w}{\langle s-r\rangle^{\frac{1}{2}+\delta}}\right\|\|\langle s+r\rangle\langle s-r\rangle^{\frac{1}{2}+\delta}\partial\partial\Gamma^{I_2}w\|_{L^\infty}\\
    		&\lesssim C_1^2\varepsilon\langle s\rangle^{-\frac{1}{2}+\delta}+C_1\varepsilon^{\frac{1}{2}}\sum_{\substack{|I_1|\leq|I|}}\left\|\frac{G\Gamma^{I_1}w}{\langle s-r\rangle^{\frac{1}{2}+\delta}}\right\|.
    	\end{aligned}
    \end{equation*}
    Applying Lemma~\ref{lem:extra} and the H\"{o}lder inequality again, we deduce
    \begin{equation*}
    	\begin{aligned}
    		\mathcal{K}_2\lesssim\sum_{\substack{|I_1|+|I_2|\leq|I|\\|J|\leq1}}\|\partial\Gamma^{I_1}w\partial\Gamma^J\Gamma^{I_2}w\|
    		\lesssim\sum_{\substack{|I_1|\leq N\\|I_2|\leq N-2}}\|\partial\Gamma^{I_1}w\|\|\partial\Gamma^{I_2}w\|_{L^\infty}\lesssim C_1^2\langle s\rangle^{-1}.
    	\end{aligned}
    \end{equation*}
    Here we have used estimates \eqref{est:commutators}, \eqref{est:L2norm1}, and \eqref{est:pointwise1}.
    
    Next, we bound the second nonlinear term. Lemma~\ref{lem:P&Q} and the H\"{o}lder inequality yield
    \begin{equation*}
    	\begin{aligned}
    		&\|\langle s+r\rangle\Gamma^IQ_0(w,w)\|
    		\lesssim\sum_{\substack{|I_1|+|I_2|\leq|I|}}\|\langle s+r\rangle G\Gamma^{I_1}w\partial\Gamma^{I_2}w\|\\
    		&\lesssim\sum_{\substack{|I_1|+|I_2|\leq|I|\\|I_2|\leq N-2}}\|\langle s+r\rangle G\Gamma^{I_1}w\partial\Gamma^{I_2}w\|\\
    		&+\sum_{\substack{|I_1|+|I_2|\leq|I|\\|I_1|\leq N-3}}\|\langle s+r\rangle G\Gamma^{I_1}w\|_{L^\infty}\|\partial\Gamma^{I_2}w\|=\mathcal{L}_1+\mathcal{L}_2.
    	\end{aligned}
    \end{equation*}
    We estimate $\mathcal{L}_1$ by considering two cases.
    
    \textbf{Case 1.} If $r\leq \frac{s}{2}$ or $r\geq 2s$, then $\langle s+r\rangle\sim\langle s-r\rangle$. By the definition of $G$, the H\"{o}lder inequality, \eqref{est:L2norm1}, and \eqref{est:pointwise1}, we get

    \begin{equation*}
    	\begin{aligned}
    		\mathcal{L}_1&\lesssim\sum_{\substack{|I_1|+|I_2|\leq|I|\\|I_2|\leq N-2}}\|\langle s-r\rangle \partial\Gamma^{I_1}w\partial\Gamma^{I_2}w\|\\
    		&\lesssim\sum_{\substack{|I_1|+|I_2|\leq|I|\\|I_2|\leq N-2}}\|\partial\Gamma^{I_1}w\|\|\langle s-r\rangle\partial\Gamma^{I_2}w\|_{L^\infty}
    		\lesssim C_1^2\langle s\rangle^{-\frac{1}{2}}.
    	\end{aligned}
    \end{equation*}
    
    \textbf{Case 2.} If $\frac{s}{2}\leq r\leq 2s$, then $\langle s-r\rangle\lesssim\langle s+r\rangle\sim\langle s\rangle$.
    \begin{equation*}
    	\begin{aligned}
    		\mathcal{L}_1&\lesssim\sum_{\substack{|I_1|+|I_2|\leq|I|\\|I_2|\leq N-2}}\left\|\frac{G\Gamma^{I_1}w}{\langle s-r\rangle^{\frac{1}{2}+\delta}}\right\|\|\langle s+r\rangle\langle s-r\rangle^{\frac{1}{2}+\delta}\partial\Gamma^{I_2}w\|_{L^\infty}\\
    		&\lesssim\sum_{\substack{|I_1|\leq|I|}}C_1\left\|\frac{G\Gamma^{I_1}w}{\langle s-r\rangle^{\frac{1}{2}+\delta}}\right\|\|\langle s-r\rangle^{\delta}\|_{L^\infty}\\
    		&\lesssim\sum_{\substack{|I_1|\leq|I|}}C_1\langle s\rangle^{\delta}\left\|\frac{G\Gamma^{I_1}w}{\langle s-r\rangle^{\frac{1}{2}+\delta}}\right\|.
    	\end{aligned}
    \end{equation*}
    
    Applying Lemma~\ref{lem:extra}, \eqref{est:L2norm1}, and \eqref{est:gamma}, we infer that
    \begin{equation*}
    	\begin{aligned}
    		&\mathcal{L}_2\lesssim\sum_{\substack{|I_1|\leq N-3, |J|=1}}C_1\|\Gamma^J\Gamma^{I_1}w\|_{L^\infty}\lesssim C_1^2\langle s\rangle^{-\frac{1}{2}+\delta}.
    	\end{aligned}
    \end{equation*}
    Therefore, combining the above estimates with~\eqref{est:conf}, we obtain
    \begin{equation*}
    	\begin{aligned}
    		&\sum_{|I|\leq N-1}\mathcal{E}_{con}^{\frac{1}{2}}(t,\Gamma^Iw)\lesssim\sum_{|I|\leq N-1}\mathcal{E}_{con}^{\frac{1}{2}}(0,\Gamma^Iw)\\
    		&+\sum_{|I|\leq N-1}\int_{0}^{t}\|\langle s+r\rangle\Gamma^IQ_P(w,w)\|\d s
    		+\sum_{|I|\leq N-1}\int_{0}^{t}\|\langle s+r\rangle\Gamma^IQ_0(w,w)\|\d s\\
    		&\lesssim C_2+C_1^2\langle t\rangle^{\frac{1}{2}+\delta}.
    	\end{aligned}
    \end{equation*}

    Based on the definitions of ghost weight energy and conformal energy, we deduce that
    \begin{align}
    	    &\|w\|\lesssim\mathcal{E}_{con}^{\frac{1}{2}}(t,w)\lesssim C_2+C_1^2\langle t\rangle^{\frac{1}{2}+\delta},\label{est:w1}\\
    		&\sum_{1\leq|I|\leq N}\|\Gamma^Iw\|\lesssim\sum_{|J|=|I|-1}\left(\mathcal{E}_{gst}^{\frac{1}{2}}(t,\Gamma^Jw)+\mathcal{E}_{con}^{\frac{1}{2}}(t,\Gamma^Jw)\right)\lesssim C_2+C_1^2\langle t\rangle^{\frac{1}{2}+\delta}.\label{est:Gamma-w1}
    \end{align}
    
    Moreover, by Klainerman-Sobolev inequality \eqref{est:Sobo}, it follows that
    \begin{align}\label{est:rough1}
    	|\Gamma^Iw|\lesssim (C_2+C_1^2)\langle t+r\rangle^{-\frac{1}{2}+\delta}\langle t-r\rangle^{-\frac{1}{2}}, \quad\mbox{for}\quad|I|\leq N-2.
    \end{align}
    Lemma~\ref{lem:extra} together with the above estimate gives
     \begin{equation}\label{est:rough2}
    	|G\Gamma^Iw|\lesssim (C_2+C_1^2)\langle t+r\rangle^{-\frac{3}{2}+\delta}\langle t-r\rangle^{-\frac{1}{2}},\quad\mbox{for}\quad|I|\leq N-3.
    \end{equation}

    \subsection{Refined bounds for conformal energy}
    In this subsection, we establish the uniform boundedness of the conformal energy by applying Lemma~\ref{lem:L2}.

 \textbf{Step 1.} Refined bounds for $\|\Gamma^Iw\|$, $|I|\leq N-1$. Based on Lemmas~\ref{lem:P&Q}, \ref{lem:extra}, and the H\"{o}lder inequality, we infer that
 
 \begin{equation*}
 	\begin{aligned}
 		&\|\Gamma^IQ_P(w,w)\|_{L^1}
 		\lesssim\sum_{\substack{|I_1|+|I_2|\leq|I|}}(\|G\Gamma^{I_1} w\partial\partial\Gamma^{I_2} w\|_{L^1}+\|\partial\Gamma^{I_1} wG\partial\Gamma^{I_2} w\|_{L^1})\\
 		\lesssim&\sum_{\substack{|I_1|+|I_2|\leq|I|\\|J|=1}}\langle s\rangle^{-1}(\|\Gamma^J\Gamma^{I_1} w\partial\partial\Gamma^{I_2} w\|_{L^1}+\|\partial\Gamma^{I_1} w\Gamma^J\partial\Gamma^{I_2} w\|_{L^1})\\
 		\lesssim&(C_2+C_1^3\varepsilon)\langle s\rangle^{-\frac{1}{2}+\delta}+C_1^2\langle s\rangle^{-1}.
 	\end{aligned}
 \end{equation*}
 Here we have used estimates \eqref{est:L2norm1}, \eqref{est:L2norm2}, \eqref{est:w1}, and \eqref{est:Gamma-w1}.
 
 Using Lemmas~\ref{lem:P&Q}, \ref{lem:extra}, and the H\"{o}lder inequality again, it follows that
 \begin{equation*}
 	\begin{aligned}
 		&\|\Gamma^IQ_0(w,w)\|_{L^1}
 		\lesssim\sum_{\substack{|I_1|+|I_2|\leq|I|}}\|G\Gamma^{I_1} w\partial\Gamma^{I_2} w\|_{L^1}\\
 		\lesssim&\sum_{\substack{|I_1|+|I_2|\leq|I|\\|J|=1}}\langle s\rangle^{-1}\|\Gamma^J\Gamma^{I_1} w\partial\Gamma^{I_2} w\|_{L^1}
 		\lesssim (C_1C_2+C_1^3)\langle s\rangle^{-\frac{1}{2}+\delta}.
 	\end{aligned}
 \end{equation*}
 Here we have used estimates \eqref{est:L2norm1}, \eqref{est:w1}, and \eqref{est:Gamma-w1}.

 Applying Lemmas~\ref{lem:P&Q}, \ref{lem:extra}, and the H\"{o}lder inequality, we deduce that
 \begin{equation*}
 	\begin{aligned}
 		&\|\Gamma^IQ_P(w,w)\|
 		\lesssim\sum_{\substack{|I_1|+|I_2|\leq|I|}}(\|G\Gamma^{I_1} w\partial\partial\Gamma^{I_2} w\|+\|\partial\Gamma^{I_1} wG\partial\Gamma^{I_2} w\|)\\
 		\lesssim&\sum_{\substack{|I_1|+|I_2|\leq|I|\\|I_1|\leq N-3}}\|G\Gamma^{I_1}w\|_{L^\infty}\|\partial\partial\Gamma^{I_2}w\|+\sum_{\substack{|I_1|+|I_2|\leq|I|\\|I_2|\leq N-3, |J|=1}}\|\Gamma^J\Gamma^{I_1}w\|\|\langle s+r\rangle^{-1}\partial\partial\Gamma^{I_2}w\|_{L^\infty}\\
 		+&\sum_{\substack{|I_1|+|I_2|\leq|I|\\|J|=1}}\|\langle s+r\rangle^{-1}\partial\Gamma^{I_1}w\Gamma^J\partial\Gamma^{I_2}w\|\\
 		\lesssim&(C_2+C_1^3\varepsilon)\langle s\rangle^{-\frac{3}{2}+\delta}+C_1^2\langle s\rangle^{-2}.
 	\end{aligned}
 \end{equation*}
 Here we have used estimates \eqref{est:L2norm1}, \eqref{est:L2norm2}, \eqref{est:pointwise1}, \eqref{est:pointwise2}, \eqref{est:Ggamma}, \eqref{est:w1}, and \eqref{est:Gamma-w1}.
 
 By Lemmas~\ref{lem:extra}, ~\ref{lem:P&Q}, \eqref{est:L2norm1}, \eqref{est:pointwise1}, \eqref{est:w1}, \eqref{est:Gamma-w1}, and \eqref{est:rough1}, we infer that
 \begin{equation*}
 	\begin{aligned}
 		&\|\Gamma^IQ_0(w,w)\|
 		\lesssim\sum_{\substack{|I_1|+|I_2|\leq|I|}}\|G\Gamma^{I_1} w\partial\Gamma^{I_2} w\|\\
 		\lesssim&\sum_{\substack{|I_1|+|I_2|\leq|I|\\|J|=1}}\|\langle s+r\rangle^{-1}\Gamma^J\Gamma^{I_1}w\partial\Gamma^{I_2}w\|\\
 		\lesssim&\sum_{\substack{|I_1|+|I_2|\leq|I|\\|I_1|\leq N-3,|J|=1}}\|\langle s+r\rangle^{-1}\Gamma^J\Gamma^{I_1}w\|_{L^\infty}\|\partial\Gamma^{I_2}w\|\\
 		+&\sum_{\substack{|I_1|+|I_2|\leq|I|\\|I_2|\leq N-2,|J|=1}}\|\Gamma^J\Gamma^{I_1}w\|\|\langle s+r\rangle^{-1}\partial\Gamma^{I_2}w\|_{L^\infty}\lesssim (C_1C_2+C_1^3)\langle s\rangle^{-\frac{3}{2}+\delta}.
 	\end{aligned}
 \end{equation*}

 Thus, for $|I|\leq N-1$, taking $\eta=\frac{2}{3}$ in Lemma \ref{lem:L2} and using~\eqref{est:conf}, we have
 \begin{equation}\label{est:GammaL2}
 	\begin{aligned}
 		\|\Gamma^{I}w\|&\lesssim\|\Gamma^Iw(0)\|+\|\partial_t\Gamma^Iw(0)\|_{L^1\cap L^2}\\
 		&+\sum_{|J|\leq |I|}\int_{0}^{t}\langle s\rangle^{-\frac{1}{3}}(\|\Gamma^JQ_P(w,w)\|_{L^1}+\|\Gamma^JQ_0(w,w)\|_{L^1})\d s\\
 		&+\sum_{|J|\leq |I|}\int_{0}^{t}\langle s\rangle^{\frac{2}{3}}(\|\Gamma^JQ_P(w,w)\|+\|\Gamma^JQ_0(w,w)\|)\d s\\
 		&\lesssim C_2+\int_{0}^{t}(C_1C_2+C_1^3)\langle s\rangle^{-\frac{5}{6}+\delta}\d s
 		\lesssim (C_1C_2+C_1^3)\langle t\rangle^{\frac{1}{6}+\delta}.
 	\end{aligned}
 \end{equation}

 By Klainerman-Sobolev inequality \eqref{est:Sobo}, we obtain
 \begin{equation}\label{est:Gamma2}
 	|\Gamma^Iw|\lesssim (C_1C_2+C_1^3)\langle t+r\rangle^{-\frac{5}{6}+\delta}\langle t-r\rangle^{-\frac{1}{2}}, \quad\mbox{for}\quad|I|\leq N-3.
 \end{equation}

 Then, Lemma~\ref{lem:extra} yields
 \begin{equation}\label{est:GGamma2}
 	|G\Gamma^Iw|\lesssim (C_1C_2+C_1^3)\langle t+r\rangle^{-\frac{11}{6}+\delta}\langle t-r\rangle^{-\frac{1}{2}}, \quad\mbox{for}\quad|I|\leq N-4.
 \end{equation}

 \textbf{Step 2.} Uniform bounds for $\|\Gamma^Iw\|$, $|I|\leq N-2$. Based on Lemmas~\ref{lem:extra},~\ref{lem:P&Q}, and the H\"{o}lder inequality, we deduce that
 \begin{equation*}
 	\begin{aligned}
 		&\|\Gamma^IQ_P(w,w)\|_{L^1}
 		\lesssim\sum_{\substack{|I_1|+|I_2|\leq|I|}}(\|G\Gamma^{I_1} w\partial\partial\Gamma^{I_2} w\|_{L^1}+\|\partial\Gamma^{I_1} wG\partial\Gamma^{I_2} w\|_{L^1})\\
 		\lesssim&\sum_{\substack{|I_1|+|I_2|\leq|I|\\|J|=1}}\langle s\rangle^{-1}(\|\Gamma^J\Gamma^{I_1} w\partial\partial\Gamma^{I_2} w\|_{L^1}+\|\partial\Gamma^{I_1} w\Gamma^J\partial\Gamma^{I_2} w\|_{L^1})\\
 		\lesssim&(C_1C_2+C_1^4\varepsilon)\langle s\rangle^{-\frac{5}{6}+\delta}+C_1^2\langle s\rangle^{-1},
 	\end{aligned}
 \end{equation*}
 in which we have used estimates \eqref{est:L2norm1}, \eqref{est:L2norm2}, and \eqref{est:GammaL2}.
 
 Using Lemmas~\ref{lem:extra},~\ref{lem:P&Q}, the H\"{o}lder inequality, \eqref{est:L2norm1}, and \eqref{est:GammaL2} again, we get
  \begin{equation*}
 	\begin{aligned}
 		&\|\Gamma^IQ_0(w,w)\|_{L^1}
 		\lesssim\sum_{\substack{|I_1|+|I_2|\leq|I|}}\|G\Gamma^{I_1} w\partial\Gamma^{I_2} w\|_{L^1}\\
 		\lesssim&\sum_{\substack{|I_1|+|I_2|\leq|I|\\|J|=1}}\langle s\rangle^{-1}\|\Gamma^J\Gamma^{I_1} w\partial\Gamma^{I_2} w\|_{L^1}\\
 		\lesssim&(C_1^4+C_1^2C_2)\langle s\rangle^{-\frac{5}{6}+\delta}.
 	\end{aligned}
 \end{equation*}
 
 Now we consider $L^2$ norms for the nonlinear terms. Lemmas~\ref{lem:P&Q}, \ref{lem:extra}, \ref{lem:commutators} together with estimates \eqref{est:L2norm1}, \eqref{est:L2norm2}, \eqref{est:pointwise1}, \eqref{est:pointwise2}, \eqref{est:GammaL2}, \eqref{est:GGamma2} give
 \begin{equation*}
 	\begin{aligned}
 		&\|\Gamma^IQ_P(w,w)\|
 		\lesssim\sum_{\substack{|I_1|+|I_2|\leq|I|}}(\|G\Gamma^{I_1} w\partial\partial\Gamma^{I_2} w\|+\|\partial\Gamma^{I_1} wG\partial\Gamma^{I_2} w\|)\\
 		\lesssim&\sum_{\substack{|I_1|+|I_2|\leq|I|\\|I_1|\leq N-4}}\|G\Gamma^{I_1}w\|_{L^\infty}\|\partial\partial\Gamma^{I_2}w\|+\sum_{\substack{|I_1|+|I_2|\leq|I|\\|I_2|\leq N-3, |J|=1}}\|\Gamma^J\Gamma^{I_1}w\|\|\langle s+r\rangle^{-1}\partial\partial\Gamma^{I_2}w\|_{L^\infty}\\
 		+&\sum_{\substack{|I_1|+|I_2|\leq|I|\\|J|=1}}\|\langle s+r\rangle^{-1}\partial\Gamma^{I_1}w\Gamma^J\partial\Gamma^{I_2}w\|\\
 		\lesssim&(C_1^4\varepsilon+C_1C_2)\langle s\rangle^{-\frac{11}{6}+\delta}+C_1^2\langle s\rangle^{-2}.
 	\end{aligned}
 \end{equation*}
 
 By Lemmas~\ref{lem:P&Q}, \ref{lem:extra}, \eqref{est:L2norm1}, \eqref{est:GammaL2}, and \eqref{est:Gamma2}, we deduce that
 \begin{equation*}
 	\begin{aligned}
 		&\|\Gamma^IQ_0(w,w)\|
 		\lesssim\sum_{\substack{|I_1|+|I_2|\leq|I|}}\|G\Gamma^{I_1} w\partial\Gamma^{I_2} w\|\\
 		\lesssim&\sum_{\substack{|I_1|+|I_2|\leq|I|,|J|=1}}\|\langle s+r\rangle^{-1}\Gamma^J\Gamma^{I_1}w\partial\Gamma^{I_2}w\|\\
 		\lesssim&\sum_{\substack{|I_1|+|I_2|\leq|I|\\|I_1|\leq N-4,|J|=1}}\|\langle s+r\rangle^{-1}\Gamma^J\Gamma^{I_1}w\|_{L^\infty}\|\partial\Gamma^{I_2}w\|\\
 		+&\sum_{\substack{|I_1|+|I_2|\leq|I|\\|I_2|\leq N-2,|J|=1}}\|\Gamma^J\Gamma^{I_1}w\|\|\langle s+r\rangle^{-1}\partial\Gamma^{I_2}w\|_{L^\infty}\lesssim (C_1^4+C_1^2C_2)\langle s\rangle^{-\frac{11}{6}+\delta}.
 	\end{aligned}
 \end{equation*}
 
 Thus, for $|I|\leq N-2$, by using~\eqref{est:conf} and Lemma~\ref{lem:L2}, and taking $\frac{1}{3}+2\delta<\eta<\frac{5}{6}-\delta$ (for example, $\eta=\frac{7}{12}$), we have
 \begin{equation}\label{est:uniformbound}
 	\begin{aligned}
 		&\|\Gamma^{I}w\|\lesssim\|\Gamma^Iw(0)\|+\|\partial_t\Gamma^Iw(0)\|_{L^1\cap L^2}\\
 		&+\sum_{|J|\leq |I|}\int_{0}^{t}\langle s\rangle^{-\frac{7}{24}}(\|\Gamma^JQ_P(w,w)\|_{L^1}+\|\Gamma^JQ_0(w,w)\|_{L^1})\d s\\
 		&+\sum_{|J|\leq |I|}\int_{0}^{t}\langle s\rangle^{\frac{7}{12}}(\|\Gamma^JQ_P(w,w)\|+\|\Gamma^JQ_0(w,w)\|)\d s\\
 		&\lesssim C_2+\int_{0}^{t}(C_1^4+C_1^2C_2)\langle s\rangle^{-\frac{9}{8}+\delta}\d s+\int_{0}^{t}(C_1^4+C_1^2C_2)\langle s\rangle^{-\frac{5}{4}+\delta}\d s
 		\lesssim C_1^4+C_1^2C_2.
 	\end{aligned}
 \end{equation}

 Based on the above estimate, we conclude that
 \begin{equation*}
 	\begin{aligned}
 		\sum_{|I|\leq N-3}\mathcal{E}_{con}^{\frac{1}{2}}(t,\Gamma^I w)
 		\lesssim\sum_{\substack{|I|\leq N-3\\|J|\leq1}}\|\Gamma^J\Gamma^I w\|\lesssim C_1^4+C_1^2C_2.
 	\end{aligned}
 \end{equation*}
    	
 In addition, by \eqref{est:Sobo} and \eqref{est:uniformbound}, we deduce
 \begin{equation*}
 	\begin{aligned}
 		\sum_{|I|\leq N-4}|\Gamma^Iw|
 		&\lesssim\langle t+r\rangle^{-1}\langle t-r\rangle^{-\frac{1}{2}}\sum_{|I|\leq N-2}\|\Gamma^Iw\|\\
 		&\lesssim (C_1^4+C_1^2C_2)\langle t+r\rangle^{-1}\langle t-r\rangle^{-\frac{1}{2}}.
 	\end{aligned}
 \end{equation*}
 
 From Lemma~\ref{lem:extra} and the above estimate, it follows that
 \begin{equation*}
 		|\partial w(t,x)|\lesssim \langle t-r\rangle^{-1}\sum_{|J|=1}|\Gamma^J w|
 		\lesssim (C_1^4+C_1^2C_2)\langle t+r\rangle^{-1}\langle t-r\rangle^{-\frac{3}{2}}.
 \end{equation*}
 Therefore, the proof of Theorem~\ref{thm:main2} is now complete.

	\section*{Data availability statement}
	No datasets were generated or analyzed during the current study.
	
	\section*{Declarations}
	\textbf{Conflict of interest:} The author declares that there are no conflicts of interest.

\end{document}